\documentclass[a4paper,leqno,12pt]{amsart}
\usepackage{tikz}
\usepackage{hyperref}
\usepackage{color,soul}
\usepackage{amsmath,amsthm}
\usepackage{amssymb}
\usepackage{esint}
\usepackage[T1]{fontenc}
\usepackage[utf8]{inputenc}
\usepackage[english]{babel}
\usepackage{amsfonts}
\usepackage{mathrsfs}
\usepackage{latexsym}
\usepackage{nicefrac}

\newcommand{\torusInt}{\fint_{\mathbb T}}
\newcommand{\D}{\mathbb D}
\newcommand{\dd}{\di}
\newcommand{\lapl}{\triangle}
\newcommand{\torus}{\mathbb T}
\newcommand{\Iverson}[1]{\bra{\,#1\,}}
\newcommand{\bigOh}[1]{\scripti O{#1}}
\newcommand{\littleOh}[1]{\mathrm o({#1})}

\newcommand{\hilbert}{\mathscr H_0}
\newcommand{\tpuo}{\mbox{$\frac 32$}}
\newcommand{\twothirds}{\mbox{$\frac 23$}}
\newcommand{\e}{\varepsilon}
\newcommand{\fii}{\varphi}

\newcommand{\bra}[1]{\left[ {#1} \right]}
\newcommand{\Bpa}[1]{\Big( {#1} \Big)}
\newcommand{\wh}[1]{\widehat {#1}}
\newcommand{\wt}[1]{\widetilde {#1}}
\newcommand{\conditionalProbability}[2]{\P(\, #1 \ehto{#2})}

\newcommand{\union}{\cup}
\newcommand{\set}[2]{\{\;#1\;|\;#2\;\}}
\newcommand{\dual}[2]{\big\langle#1,#2\big\rangle}

\newcommand{\puokki}{\mbox{$\frac12$}}
\newcommand{\smallFrac}[2]{\mbox{$\frac{#1}{#2}$}}

\newcommand{\expectation}{\EE}

\newcommand{\bpa}[1]{\big( #1 \big)}
\newcommand{\brc}[1]{\left\{ #1 \right\}}
\newcommand{\ehto}[1]{\, | \, #1 \,}
\renewcommand{\P}{\text{{\bf P}}}
\newcommand{\scripti}[1]{{\mathscr {#1}}}
\newcommand{\scr}[1]{{\mathscr {#1}}}
\newcommand{\R}{\mathbb{R}}
\newcommand{\C}{\mathbb{C}}
\newcommand{\N}{\mathbb{N}}

\newcommand{\Z}{\mathbb{Z}}
\newcommand{\di}{\mathrm{d}}
\newcommand{\EE}{\mathbf{E}}
\newcommand{\abs}[1]{\left\lvert #1 \right\rvert}

\newcommand{\norm}[1]{\left\| #1 \right\|}

\newtheorem{thm}{Theorem}[section]
\newtheorem{lem}[thm]{Lemma}
\newtheorem{prop}[thm]{Proposition}

\newtheorem{cor}[thm]{Corollary}

\newtheorem*{notation}{Notation}
\theoremstyle{definition}

\theoremstyle{remark}
\newtheorem*{remark}{Remark}

\newcommand{\FBN}{Fractional Brownian noise}

\title[FBM and asymptotic Bayesian estimation]{Fractional Brownian motion and asymptotic Bayesian estimation}

  \def\oldSymbol#1{f_{#1}}
  \def\Sinai{Sina\u\i}

  \def\vec#1{\mathbf {#1}}
  \def\ej#1{{\vec e^{#1}}}

  \def\inverse#1{#1^{-1}}

  \def\wh{\widehat}
  \def\puo{\puokki}
\begin{document}

\author[Lassi Päivärinta]{Lassi Päivärinta}
\address{Tallinn University of
Technology\\Department of Mathematics,\linebreak Ehitajate tee 5\\19086 Tallinn, Estonia}
\email{lassi.paivarinta@ttu.ee}

\author{Petteri Piiroinen}
\address{%
University of Helsinki\\Department of
Mathematics and Statistics\\P.O. Box 68\\FI-00014 University of
Helsinki, Finland
}
\email[Corresponding author]{petteri.piiroinen@helsinki.fi}

\subjclass[2010]{60G22, 60F15, 60F17, 60G15, 62F15, 62F12, 47B35}

\begin{abstract}
In this paper, we study the recovery of the Hurst parameter from a given
discrete sample of fractional Brownian motion with statistical inverse
theory. In particular, we show that in the limit the posteriori
distribution of the parameter given the sample determines the parameter
uniquely. In order to obtain this result, we first prove various strong
laws of large numbers related to the problem at hand and then employ these
limit theorems to verify directly the limiting behaviour of posteriori
distributions without making additional technical or simplifying
assumptions that are commonly used.
\end{abstract}
\maketitle

\section{Introduction}
We study the recovery of the Hurst parameter from a given discrete sample
of fractional Brownian motion with statistical inverse theory. In
particular, we show that in the limit the posteriori distribution of the
parameter given the sample determines the parameter uniquely.

 Fractional Brownian motion $Z^H$ is a one parameter generalization
 of the standard Brownian motion $B$ introduced in \cite{MvN:68}. The
 generalization corresponds to changing the variance function
 $\mathbf{V} B_t = |t|$ to the variance function $\mathbf{V} Z_t =
 |t|^{2H}$ where the Hurst parameter $H \in (0,1)$. It can be shown that
 with this choice the fractional Brownian motion exists as a
 stochastically continuous Gaussian process with stationary increments
 (c.f.~e.g.~\cite{doukhan2003theory}).
 These increments are not, however, independent unless $H = \frac 12$
 which corresponds to the Brownian motion case. This makes the analysis
 of these processes more involved.

 The inverse problem we have in mind is the usual parameter estimation
 problem. We sample a given signal $X^{\widehat H}$ at equidistant
 time instances $t_j = j/n$ for every $j=0,1,\dots,n$. From this data we
 form the increments $Y^{\widehat H}_j := X^{\widehat H}(t_j) -
 X^{\widehat H}(t_{j-1})$. The reason for using the increments is
 motivated by the stationarity.

 We formulate the parameter estimation problem as the Bayesian
 estimation problem: 
 \begin{quote}
 \emph{Determine the conditional probability distribution
 of $H$ given the sample $(Y^{\widehat H}_1, \dots, Y^{\widehat
 H}_n)$. From the conditional probability distribution construct
 estimators for the true parameter $\widehat H$.}
 \end{quote}
 We solve this using the standard Bayesian methods with the assumption
 that the prior is noninformative. This leads to the solution of form
 \[
 \mathbf P( H \in U \, | \, \xi_{\widehat H,n} )
 = C_n(\xi_{\widehat H,n}) \int_U
 \frac{n^{nu}}{\sqrt{| T_n(f_u) |}} 
 e^{-\frac 12 n^{2(u - \widehat H)}
 Q_n(\xi_{\widehat H, n}, u)} \mathrm d u
 \]
 where $T_n(f)$ is the $n \times n$ Toeplitz matrix corresponding to the
 symbol $f$, the $|\,\cdot\,|$ stands for the determinant and where
 $Q_n$ is the quadratic form 
 \[
 Q_n(y, u) = \langle y , T_n(f_u)^{-1} y \rangle.
 \]
 The data $\xi_{\widehat H,n}$ we use in the posteriori solution is the
 rescaled increments $\xi_{\widehat H, n} := n^{\widehat H} (Y^{\widehat
 H}_1,\dots, Y^{\widehat H}_n)$. The symbol $f_u$ corresponds to the
 covariance operator of the increments of fractional Brownian motion
 $Z^u$ sampled at integer points.

 This solution provides a numerical reconstruction method.
 However, the computational cost of numerically calculating the
 quadratic form $Q_n$ is both expensive and quite unstable.
 Moreover, the computation of the Toeplitz determinant is likewise
 rather expensive and unstable. Therefore, only relatively small values
 of $n$ can be used in numerical analysis and for small $n$ the
 reconstruction from a simulated data does not appear to be very
 consistent with the true parameter value.

The estimation of the Hurst parameter of fractional Brownian motion from the
measured data has a vast literature. It has been applied, for instance, to
estimation in financial markets when we assume that there is a long term
memory effects \cite{MR1920167,MR2130350, MR1849425} and river
flows~\cite{hurst1951long, millen2003estimation}. The name of the
Hurst parameter comes from the study of River Nile by
Hurst~\cite{hurst1951long} in 1951.

 The question we wanted to analyse is how these distributions behave
 asymptotically as the number $n$ of samples tends to infinity. 

This problem and related questions goes back at least 65 years and
has an extensive literature.
The first approach seems to be the line of estimation which 
could be called the methods of \emph{Whittle approximation} type. This
was introduced by Whittle~\cite{Whittle:51} in 1951 and it has been applied for many
different kinds of approximations for the maximum aposteriori (MAP)
estimate (c.f.~\cite{MR3252632, MR2000409,  MR2578831, MR0056902,
MR1901156,  MR2498729, MR2189074, MR2589196, MR2210688}). This
corresponds to replacing the inverse matrix $T_n(f_u)^{-1}$ by the
Toeplitz matrix $T_n(1/f_u)$.

The first significant contribution to the estimation problem for
fractional Brownian motion following this line of reasoning came from Taqqu and
Fox~\cite{FoxTaqqu:86}. 
In~\cite{FoxTaqqu:86} Fox and Taqqu showed that the estimator
corresponding to the Whittle approximation (without the determinant) is
asymptotically normal and gives the true parameter value $\widehat H$ in
the limit \emph{provided} that $\widehat H > \frac 12$. 
Afterwards the estimation result of Fox and Taqqu has been used and
sharpened by various authors (see for instance
~\cite{Beran:94,LeuPapamarcou:95}).

 The limitation $\widehat H > \frac 12$ has been present in most of these
 studies and it has only been removed by different kinds of
 estimators~\cite{BerzinLeon:07, MR1804246}. This limitation comes from the
 techniques used. However, the more profound limitation is
 that of using the Whittle approximation. The authors have not found
 any results that would give reasons to believe that this approximation
 would be a good approximation for the original problem, even though it
 might be superior way for estimating the parameters numerically. The heuristics and referred articles in
 \cite{Beran:94,Blekher:79} are only concerned of the inversion of
 Toeplitz matrices and not of the mapping properties of the perturbation
 caused by the using an approximation. During the preparation of this
 manuscript, the authors found that the perturbation is not a classical
 (compact) perturbation even in the correct scale of function spaces.
 This result is, however, omitted from this manuscript and will be part
 of a later study. Naturally, these approximations were done to make the
 numerics faster and they work nicely and even with finite band
 approximation~(see~\cite{Barbara:09}).
 
 In the present article, we tackle the original problem without using
 any approximations. The non-approximative results for the MAP estimator
 have been considered by Dahlhaus~(\cite{Beran:94, MR1026311,
 MR2283403}), but as far as the authors are aware, the estimation for
 the whole posteriori distribution is not done before. We show 
 that the posteriori distribution converge weakly to point
 mass on top of $\widehat H$ almost surely (Corollary~\ref{seuraus12}).
 This follows from the characterization result for the asymptotic
 conditional distribution (Theorem~\ref{lause11}). More precisely, we show 
 that with probability one the conditional distribution of $\widetilde
 H_n$ is asymptotically standard normal with mean $\alpha_n$ where the random
 variable $\widetilde H_n$ is a rescaled version of the random variable
 $H$. As a part of this result we deduce that the asymptotic variance of
 $H$ around the mean $\alpha_n$ is of order $n^{-1} (\log n)^{-2}$
 (Lemma~\ref{main}).
 Furthermore, we show that the means $\alpha_n$ converge to $\widehat H$
 as $n \to \infty$ with a rate of order $(\log n)^{-2}$
 (Lemma~\ref{mean}). 

 From this asymptotic normality we can read that the usual
 estimators like conditional mean and the MAP estimates are biased for
 large $n$ but they are asymptotically unbiased and asymptotically exact.
 
 In the proof of the asymptotic parameter estimation result we use the
 asymptotics for Toeplitz determinants with one Fisher--Hartwig
 singularity \cite{ES:97}.  In order to handle the randomness coming from
 the random quadratic form, we prove the (uniform) \emph{Strong Law of Large
 Numbers (SLLN)} for the quadratic form $Q_n$ appearing in the 
 posteriori distribution. This uniform SLLN (Theorem~\ref{thm3}) is
 the most involved part of the article since it builds upon the previous
 two Theorems (Theorem~\ref{thm1} and~\ref{thm2}) and proving it
 requires different techniques from probability theory, some ideas from
 the theory of Toeplitz operators and asymptotics for the inverse
 Toeplitz matrices \cite{RambourSeghier2, RambourSeghier}.

The rest of the paper is structured as follows: We start in Section
\ref{Notations} by briefly introducing our notation. In Section
\ref{SLLNGQF}, we prove a generic strong law of large numbers for a
sequence of Gaussian quadratic forms (Theorem~\ref{thm1}). This result
forms the basis for the rest of the asymptotic results. There are many
special cases of the result in the literature, but as far as the authors
are aware, the result of Theorem~\ref{thm1} is novel.

In Section \ref{SLLN}, we show the main pointwise strong law of large
numbers for the sequence of quadratic forms arising from the finite
samples of fractional Brownian motion (Theorem~\ref{thm2}). The main
novelty is that with this result we obtain for which Hurst parameter
values the random quadratic forms have almost sure limits and for what
it diverges. 
This enables us to remove the usual
technical restrictions of $H > \puokki$. The proof relies on the earlier results
of Rambour and Seghier~\cite{RambourSeghier2, RambourSeghier}. In our
case, the Toeplitz symbol $g_\alpha$ (which we will introduce in
Section~\ref{Notations}) does not satisfy
the assumptions of the main theorems in~\cite{RambourSeghier2, RambourSeghier}
for every Hurst parameter value. Therefore, we have to
generalize the results of~\cite{RambourSeghier2, RambourSeghier} for
slightly larger class of symbols. To do so we deduce in the
Section~\ref{factorisation} few lemmata that provide the needed
factorisation of the Fisher--Hartwig symbol. For these we follow the
techniques of Grenander and Szeg{\H o}~\cite{GrenanderSzego}. Since the
proofs are mainly technical lemmata, the proofs are postponed to the
appendix.

Subsequently, in Section~\ref{USLLN} we improve the pointwise strong law
of large of the quadratic forms arising from the finite samples of
fractional Brownian motion into a functional strong law of large numbers
which we call as the Uniform Law of Large Numbers (Theorem~\ref{thm3}).
This result implies that posteriori distributions of the Hurst
parameters given the finite samples of FBM converge weakly almost surely as the
sample size grows to infinity. The proof relies on Helly's Selection
Theorem and analysis of the Fisher--Hartwig singularity of the symbols
$g_\alpha$ together with the pointwise Strong Law of Large numbers
(Theorem~\ref{thm2}).

Subsequently, in Section
\ref{param-estimation}, we use the Theorem~\ref{thm3} together with
simple asymptotic analysis to derive the asymptotic behaviour of the
sequence of posteriori distributions of the unknown Hurst parameter $H$
given the finite samples of FBM (Theorem~\ref{lause11}.
Finally, the Appendix is divided into
Sections~\ref{PFA},~\ref{PFB},~\ref{PFC},~\ref{PFD} and~\ref{PFE} that consists the proofs of auxiliary lemmata of
Sections~\ref{SLLNGQF},~\ref{SLLN},~\ref{factorisation},~\ref{USLLN},
and~\ref{param-estimation}, respectively.

\section{Notations}
\label{Notations}
For the reason of notational compactness we use the Iverson brackets in this
paper. Since it is a rather atypical notation in the field, we introduce it properly.
\begin{notation}[Iverson bracket]
The notation $\left[\cdot\right]$ is the Iverson bracket (see for example
\cite{Knuth})
\begin{equation*}
 \left[\,A\,\right] :=
   \begin{cases}
      1,& \text{if $A$ is true,} \\
      0,& \text{otherwise.}
   \end{cases}
\end{equation*}
We also use the Iverson brackets to denote the indicator functions by
notation
\begin{equation*}
\left[\,A\,\right](x) := \left[\,x \in A\,\right].
\end{equation*}
\end{notation}
The benefit of this is that we can then use the standard trick of
probability theory to eliminate the elementary events from expectations.
For example, we can write $\mathbf{E} X \left[\,A\,\right]$ instead of
the more cumbersome notations
\begin{equation*}
\mathbf{E} X \left[\, \cdot \in A \,\right] = \int_\Omega X(\omega)
\left[\, \omega \in A\,\right] \mathbb{P}(\mathrm{d} \omega\,) = \int_A
X(\omega) \mathbb{P}(\mathrm{d} \omega\,).
\end{equation*}

Because of this, we will never use square brackets as an alternative
to parenthesis. The only case, where we use square brackets and not mean
the indicator function is when we denote closed intervals. However,
these cases are easily recognised from the context.

Toeplitz matrices and operators have a significant role in this article.
Toeplitz matrices and operators are in a close relation with circulant
matrices and convolution operators. In this article, we mean by a convolution operator
an infinite dimensional matrix index over integers that forms a Fourier
pair with a multiplication by a \emph{symbol} acting on smooth periodic functions on top of
torus $\torus$.
We denote the convolution operator and the symbol as mappings
$C(g_\alpha)\colon c_{00} \to {\C^\Z}$, 
\[
\dual{C(g_\alpha) {\ej j}}{\ej k} := c(g_\alpha) (j-k)
:= \torusInt e^{-i(j-k)t}  g_\alpha (t) \di t
\]
where ${g_\alpha}$ is the \emph{symbol} or \emph{spectral density
function} and
$c_{00}$ stands for the sequences with only finitely many nonzero elements.
The importance of these convolution operators stem from the spectral
representation for the fractional Brownian noise.
Yakov G.~Sina\u\i\ has shown in 1976 \cite{Sinai:76} the following fact.
\begin{lem}
  \label{lemma1}
    The spectral density function of \FBN\ has a representation
    \begin{equation}
      \label{eq5}
      f_H(\lambda) = C_H \abs{e^{i\lambda} - 1}^2 
	\sum_{k \in \Z} \abs{\lambda - 2\pi k}^{-(2H + 1)},
    \end{equation}
    where $C_H \in \R$ is a norming constant.
\end{lem}
\begin{proof}
  See~\cite{Sinai:76}.
\end{proof}
For symmetry reason we use $g_\alpha := \oldSymbol {\puokki + \alpha}$ with
$-\puokki < \alpha < \puokki$. This is since by \Sinai's result we have
\[
\begin{split}
\forall \alpha \in (0,1)\colon\oldSymbol\alpha(t)\Iverson{\abs t < \e} 
    &= c_\alpha t^2 ( 1 + \bigOh t^2) (\abs t^{-(2\alpha+1)} + \bigOh 1)
    \Iverson{\abs t < \e}\\
    &= c_\alpha \abs t^{1-2\alpha} ( 1 + \littleOh t) 
    \Iverson{\abs t < \e}.
\end{split}
\]
Rewriting this for the new symbol yields
\[
\begin{split}
\forall \alpha \in (-\puokki,\puokki)\colon{g_\alpha}(t)\Iverson{\abs t < \e} 
    &= \oldSymbol{\puokki+\alpha}(t)\Iverson{\abs t < \e} \\
    &\asymp \abs t^{-2\alpha} \Iverson{\abs t < \e},
\end{split}
\]
which reveals the simple relation between the symbol ${g_\alpha}$ and the
order of the zero at origin.

Usually, the we only choose some part of the convolution operator
$C(f)$ and these are the Toeplitz operators $T(f)$ and Toeplitz matrices
$T_n(f)$. In this work, we don't really need the infinitedimensional
operators $C(f)$ and $T(f)$, since the mapping properties are not nice
enough in this setting, so we only really need the Toeplitz matrices
$T_n$ which are defined as
\[
T_n(g_\alpha) = \big(C(g_\alpha)_{jk}\big)_{j,k=1}^n =
\big(c(g_\alpha)(j-k)\big)_{j,k=1}^n
\]
These Toeplitz operators and matrices correspond to the convolution
operator acting on analytic functions projected to the analytic
functions. The mappings that correspond to the convolution operator
acting on analytic function but projected to the anti-analytic functions
are called Hankel operators and Hankel matrices. While these kind of
mappings are really needed in this work, we, however, only need them
implicitly (see Lemma~\ref{lemma:B:L2}) and therefore, we don't give the
actual definitions.

In several occasions, we need to represent a function in a point free
manner, so we use notation $f = x \mapsto f(x)$ to denote that $f$ is a
mapping and $f(x)$ is its value. For a curried function $f = x \mapsto (y
\mapsto F(x,y))$, the $f(x) = y \mapsto F(x,y)$ and $f(x)(y) = F(x,y)$.
When we don't need to give a name for a function, we use $x \mapsto
\dots$ to denote the anonymous function.

In Section~\ref{SLLNGQF} and especially in Section~\ref{PFA}, we use
extensively Frobenius norm, Frobenius inner product and tensor products
of matrices. We denote the Frobenius inner product of two matrices $A$
and $B$ by $A : B$ and this is defined as
\[
A : B = \sum_{jk} A_{jk}B_{jk} = 
\text{Tr}\,(AB^\top)
\]
The Frobenius norm of a square matrix $A$ is $\norm{A}_F = \sqrt{A:A}$.
We interpret multi-indices as words, i.e.\ for a matrix $A =
(A_{jk})_{jk}$ we interpret $\rho = jk$ as a word of two letters. The
words $\rho$ and $\eta$ can be concatenated with concatenation operation
$\rho \& \eta$ which is a word (or multi-index) obtained by joining the
two words. For example, if $\rho = jk$ and $\eta = lm$ then $\rho \&
\eta = jklm$.

We will use tensor notation to deal with multilinear objects that we
arrive when computing higher moments and we identify vectors and
matrices with  tensors with one letter words and tensors with two letter
words, respectively. In this work we will mean by (covariant) tensors
the mappings from words (i.e.  elements of $\N^* = \N \union \N^2 \union
\dots$) to scalar field, i.e.  $A = \rho \mapsto A_\rho$. The tensor
product of two tensors $A$ and $B$ is defined as $A \otimes B = (\rho \&
\eta) \mapsto A_\rho B_\eta$. The tensor power $A^{\otimes n}$
is the $n$-fold tensor product $A \otimes \dots \otimes A$

In Section~\ref{PFA} and for instance in the claim of
Lemma~\ref{lemma:L8}, we denote the falling product by $a^{\underline
n}$. The falling product is defined as $a^{\underline n} = a (a-1) \dots
(a-n+1)$. In Section~\ref{SLLNGQF} we use for the first time the lattice
operations $\wedge$ and $\vee$ to denote the minimum and maximum, respectively.
The multi-index power means the usual $x^{\rho} = x_1^{\rho_1}
x_2^{\rho_2}
\dots$ where the vector $x$ and the multi-index $\rho$ share the same
finite dimension.

Throughout the work, we denote majorization by $f \lesssim g$, by which
we mean that there is a positive constant $c > 0$ such that $f \le c g$.
Similarly, $f \asymp g$ is $f \lesssim g \lesssim f$.
We typically write these in a pointed manner, i.e. as $f(x) \lesssim
g(x)$ and by context it should be clear which argument we are the
considering. Moreover, the domain where this majorization is concerned
is usually some neighbourhood of some infinity point, but this should be
clear from the context. Few times we denote $f \ll g$ instead of $f =
\littleOh g$.

We will use ellipsis (i.e. ``$\dots$'') to denote something that we decided to
temporarily omit writing.
This is typically used together with integration where we temporarily
don't write the integrand explicitly. Furthermore, when we write
\[
\torusInt \dots \dd t 
\]
we mean integration with respect to the normalized Lebesgue measure 
on a torus $\torus = [-\pi,\pi)$. We will call the interval $[-\pi,\pi)$
as a torus even though we don't explicitly map it to a unit circle on a
plane.

\section{Strong law of large numbers for symmetric Gaussian quadratic
forms}
\label{SLLNGQF}
In this section we prove an auxiliary limit result that we need for the later
limit theorems. We will denote by $(A_n)$ a sequence of symmetric
matrices in $\R^{n\times n}$ and assume that $(\xi_n)$ is a sequence of
Gaussian random variables with zero mean and covariance matrices $C_n
\in \R^{n\times n}$.
\begin{thm}
  \label{thm1}
    Suppose $0 \le \gamma < 1$. If $\norm{C_n^{\nicefrac1{2}}A_nC_n^{\nicefrac1{2}}}_F \lesssim n^\gamma$,
    then 
    \[
    \lim_{n\to \infty} n^{-1} \big(\dual{\xi_n}{A_n\xi_n}
    -\expectation\dual{\xi_n}{A_n\xi_n}\big) = 0
    \]
    almost surely.
\end{thm}

In the latter part of this section, we will fix $n$ and therefore, we
will drop the subscript $n$ to simplify notations and to release $n$ for
other uses. We will denote
\begin{equation}
  \label{eq:D3}
    \Theta(B,n) = \sum_\rho B\otimes A^{\otimes (n-1)}_\rho \expectation
    \xi^{\otimes 2n}_\rho
    \quad \text{and}\quad \Theta(n) = \Theta(A,n).
\end{equation}
We note that $\Theta(n) = \expectation \dual{\xi}{A\xi}^n$.
Furthermore, we will denote
\begin{equation}
  \label{eq:D4}
    R_j = \text{Tr}\,(AC)^j\quad \text{and}\quad R^\mathbf k = \prod_j
    R_{\mathbf k_j}
\end{equation}
for every $\mathbf k \in \N_+^*$. We will need to have a control for the
multi-indices $\mathbf k$ so we define few sets of multi-indices. First
we need to know the number of ones in a multi-index. We denote the
counting function by $\theta$, and the cumulative functions by $s_k$
\begin{equation}
  \label{eq:D5}
    \theta(\mathbf k) = \sum_j \Iverson{\mathbf k_j = 1}
    \quad\text{and}\quad
    s_k(\mathbf k) = \sum_j \mathbf k_j \Iverson{1 \le j \le k}.
\end{equation}
With the counting and cumulative functions we define
\begin{equation}
  \label{eq:D6}
  \begin{split}
    J(m,n) &= \brc{\mathbf k \in \N_+^m \,;\, s_m(\mathbf k) = n }
    \quad \text{and}\\
    J_l(m,n) &= \brc{\mathbf k \in J(m,n) \,;\, \theta(\mathbf k) = l }.
  \end{split}
\end{equation}
To each multi-index $\mathbf k \in J_0(n,m)$ we associate the following
number
\begin{equation}
  \label{eq:D8}
    \log c(\mathbf k) = -\sum_{j=2}^n \log
    \big(s_n(\mathbf k) - s_{j-1}(\mathbf k) \big) +
    \sum_{j=1}^{s_n(\mathbf k) - 1} \log j
\end{equation}
We still need one auxiliary function so that we can formulate the
representation lemma for the cumulant functions. We denote the cumulant
function by $\Psi$
\begin{equation}
  \label{eq:D7}
    \Psi(N) = \sum_{n=0}^N \binom Nn (-1)^{N-n} \Theta(n) R_1^{N-n}.
\end{equation}
We can now formulate the representation result.
\begin{lem}
  \label{lemma:L7}
   For $N > 0$ we have
   \[
   \Psi(N) = \sum_{m=1}^N 2^{N-m} \sum_{\mathbf k} \Iverson{\mathbf k
   \in J_0(m,N)} R^{\mathbf k} c(\mathbf k)
   \]
\end{lem}
This representation has two important aspects that are
that each multi-index that appears on the right-hand side has
$s(\mathbf k) = N$ and $\theta(\mathbf k) = 0$. The proof of
Lemma~\ref{lemma:L7} is given in the end of this secion.

We will start by proving the strong law.
\begin{proof}[Proof of Theorem~\ref{thm1}]
According to the assumption, the values $R_j = R_j(n)$ satisfy
\[
R_{2j}(n) \le \norm {C_n^{1/2}A_nC_n^{1/2}}^{2j}_F \lesssim n^{2j \gamma}
\]
and by Cauchy--Schwarz inequality for Frobenius inner product
\[
R_{2j+1}(n) \le \norm {C_n^{1/2}A_nC_n^{1/2}}^{2j}_F
\norm{C_n^{1/2}A_nC_n^{1/2}}_F \lesssim
n^{(2j+1)\gamma}.
\]
Let us denote
\[
X_n = \dual {\xi_n}{A_n\xi_n} - \expectation \dual {\xi_n}{A_n\xi_n}.
\]
The estimates for $R_j(n)$ combined with Lemma~\ref{lemma:L7} gives
\[
\expectation{X_n^{2N}} \lesssim n^{2N\gamma}.
\]
Therefore, if we choose $N > (1-\gamma)^{-1}$ we see that
\[
\expectation (n^{-1}X_n)^{2N} \lesssim n^{-2}.
\]
This estimate together with application of Chebysev Inequality implies that
\[
\sum_{n} \P(n^{-1}|X_n| > \varepsilon) < \infty
\]
for every $\varepsilon > 0$.
The claim of the Theorem follows immediately from this by Borel--Cantelli Lemma.
\end{proof}
\begin{remark}
  The proof of Theorem~\ref{thm1} relies heavily on the representation
  given by
  Lemma~\ref{lemma:L7}. It is straightforward to show a similar
  representation but without the extra condition $\theta(\mathbf k) = 0$
  for every $\mathbf k$ appearing on the right. If we suppose that 
  $\lim n^{-1}\expectation \dual{\xi_n}{A_n\xi_n} = \lim n^{-1} R_1(n) =
  F(A,C) > 0$, then terms $R_1(n) \asymp n$. In worst case terms with
  $\theta(\mathbf k) \asymp N$ would prevent obtaining the convergence
  result for $\gamma$ sufficiently close to $1$.
\end{remark}
We first compute the first representation formula for $\Theta(n)$ which
is a simple recursive formula.
\begin{lem}
  \label{lemma:L8}
    We have for $n \ge 1$
    \[
    \Theta(n) = \sum_{j=1}^n 2^{j-1} (n-1)^{\underline{j-1}} R_j
    \Theta(n-j)
    \]
\end{lem}
Using this we obtain the first exact representation. For this we denote
\begin{equation}
  \label{eq:D9}
    \Theta(n) = \sum_{j=k}^n a_k(j,n) \Theta(n-j)
\end{equation}
for every $k \le n$. We note that $a_1$ is defined by
Lemma~\ref{lemma:L8}, moreover applying the same lemma gives recursive
formula for $a_k$.
\begin{lem}
  \label{lemma:L9}
    We have for $1 \le k < j \le n$ that
    \[
    a_{k+1}(j,n) = a_k(k,n)a_1(j-k,n-k) + a_k(j,n).
    \]
\end{lem}
Since $\Theta(n) = a_n(n,n)$, solving the recursion equation for $a_k$ solves $\Theta$ as well.
\begin{lem}
  \label{lemma:L13}
    We have for $n \ge 1$ that
    \[
    \Theta(n) = \sum_{m=1}^{n} \sum_{\mathbf k \in J(m,n)} 2^{n-m}
    R^{\mathbf k} \prod_{j=1}^{m} (n-s_{j-1}(\mathbf k)-1)^{\underline{\mathbf k_j -
    1}} 
    \]
\end{lem}
The value $R^{\mathbf k}$ is invariant with respect to permutations
of $\mathbf k$. Furthermore, we note that whenever $\mathbf k_j = 1$ the
term satisfies
\[
(n-s_{j-1}(\mathbf k)-1)^{\underline{\mathbf k_j - 1}} = 1.
\]
Moreover, $J(m,n) = \bigcup_l J_l(m,n)$ where $J_0(m,n)$ will be
called the good part and the rest as the bad part. The goal is to show
that in the end the bad part cancels out, so we need to have explicit
division in to these two parts.  For this we introduce a new set
of multi-indices
\[
L(m,n) = \brc{1,\dots,n}^{m} \cap \text{INC}
\]
where $\text{INC}$ denotes strictly increasing sequences of any length.
In this way we can divide $J(m,n)$ in to two parts, to part consisting
of $1$'s and the rest. Therefore, we denote for every $\mathbf k \in
J_0(m-l,n-l)$ and every $\mathbf \lambda \in L(m-l,m)$
\begin{equation}
  \label{eq:D12}
  \pi_j(\mathbf k, \mathbf \lambda) = 1 + \sum_l (\mathbf k_l -
  1)\Iverson{\mathbf\lambda_l = j}.
\end{equation}
We note that $\pi(\mathbf k, \mathbf \lambda) \in J_l(m,n)$ and every
element in $J_l(m,n)$ is obtained in the process.

The last remaining part is the auxiliary function is
\begin{equation}
  \label{eq:D14}
\Lambda(n,\mathbf k) =
  \sum_{\lambda \in L(m-l,m)} \prod_{j=1}^{m-l} (n-s_{\mathbf
  \lambda_j-1}(\pi(\mathbf k, \lambda)) -
  1)^{\underline{\mathbf k_j - 1}}
\end{equation}
whenever $\mathbf k \in J_0(m-l,n-l)$.
With the help of these notation we can give the next representation for
$\Theta$.
\begin{lem}
  \label{lemma:L15}
    We have for $n \ge 1$ that
    \[
    \Theta(n) = \sum_{m=1}^{n} \sum_{l=0}^{m} 2^{n-m} R^l_1
    \sum_{\mathbf k}
    R^{\mathbf k} \Lambda(n, \mathbf k)
    \Iverson{\mathbf k \in J_0(m-l,n-l)}.
    \]
\end{lem}
It turns out that there is a simple representation for $\Lambda$.
\begin{lem}
  \label{lemma:L16}
    We have for every $n \ge 1$ and $\mathbf k \in J_0(m,n-l)$ that
    \[
    \Lambda(n, \mathbf k) = c(\mathbf k) \binom n{n-l}.
    \]
\end{lem}
We can now prove the main representation lemma (Lemma~\ref{lemma:L7}).
\begin{proof}[Proof of Lemma~\ref{lemma:L7}]
Combining Lemmata~\ref{lemma:L15} and~\ref{lemma:L16} we have by
using little algebra that when $n \ge 0$ 
\[
\begin{split}
R_1^{N-n} \Theta(n) &= R_1^N + \sum_{M=1}^n R_1^{N-M} \sum_{m
\ge 1} 2^{M-m} \sum_{\mathbf k \in J_0(m,M)} R^{\mathbf k} c(\mathbf k)
\binom nM \\
&=\sum_{M=0}^n R_1^{N-M} \kappa(M) \binom nM
\end{split}
\]
 Since $\Theta(0) = 1$ we obtain that
\[
\Psi(N) = \sum_{M=0}^N \kappa(M) \sum_{n=M}^N \binom Nn
(-1)^{N-n} \binom nM = \kappa(N)
\]
and the claim follows.
\end{proof}
\section{Pointwise strong law of large numbers for FBN}
\label{SLLN}
We will apply the Theorem~\ref{thm1} in order to obtain a pointwise SLLN
for FBN.
\begin{thm}
  \label{thm2}
    Suppose $\alpha_- + \beta_+ < \puokki$. Suppose $\xi_n \sim
    N(0,T_n(g_\beta))$ for every $n$. Then 
    \[
    \lim_{n\to \infty} n^{-1}\dual{\xi_n}{T_n(g_\alpha)^{-1}\xi_n} =
    \torusInt \frac{g_\beta(t)}{g_\alpha(t)} \di t
    \]
    almost surely.
\end{thm}
For this we need few propositions and couple of lemmata. We will
postpone the proofs of these results to the Appendix
(Section~\ref{PFB}). First
proposition states that the Theorem~\ref{thm1} is applicable in our
case.
\begin{prop}
  \label{lemma:B:L1}
    For every $\alpha, \beta \in (-\puokki, \puokki)$ we have
    \[
    \norm{T_n(g_\beta)^{\nicefrac1{2}}T_n(g_\alpha)^{-1} T_n(g_\beta)^{\nicefrac1{2}}}_F
    \asymp n^{2(\alpha_- + \beta_+) \vee \nicefrac1{2}}
    \]
\end{prop}
Next we need to express the inverse Toeplitz matrix as a perturbation of
a Toeplitz matrix.
\begin{lem}
  \label{lemma:B:L2}
    For every $\alpha \in (-\puokki, \puokki)$ there exists a matrix
    $K_n(\alpha)$ such that
    \[
    T_n(g_\alpha)^{-1} = T_n(g^{-1}_\alpha) +
    \puokki\big(T_n(g_\alpha)^{-1}K_n(\alpha) + K_n(\alpha)^*
    T_n(g_\alpha)^{-1}\big)
    \]
\end{lem}
Let us denote $\widetilde K_n(\alpha) = \puokki\big(|K_n(\alpha)| +
|K_n^*(\alpha)|\big)$. We note that when $\alpha = 0$, the matrix
$\widetilde K_n(0) = 0$. For the matrix $\widetilde K_n(\alpha)$ we
state the following properties.
\begin{prop}
  \label{lemma:B:L3}
    For every $\alpha \in (-\puokki, \puokki)$ 
    \[
    |T_n(g_\alpha)^{-1}| \widetilde K_n(\alpha) \lesssim \Iverson{\alpha
    \ne 0} n^{2\alpha_-} \widetilde K_n(\alpha)
    \]
\end{prop}
Furthermore, we still need one estimate
\begin{prop}
  \label{lemma:B:L4}
    For every $\alpha \in (-\puokki, \puokki)$ and $\beta \in (-\puokki,
    \puokki)$ we have
    \[
    |T_n(g_\beta)| : \widetilde K_n(\alpha) \lesssim
    \Iverson{\alpha \ne 0} (n^{2\beta} \vee \log n)
    \]
\end{prop}
These propositions are essential pieces for proving the SLLN for
fractional Brownian noise.
\begin{proof}[Proof of Theorem~\ref{thm2}]
 Let us denote $A_n = T_n(g_\alpha)^{-1}$, $B_n = T_n(g^{-1}_\alpha)$ and $C_n = T_n(g_\beta)$.
 Furthermore, we will drop $\alpha$ from $K_n(\alpha)$ and $\widetilde
 K_n(\alpha)$ since $\alpha$ is fixed.
 The assumption together with Proposition~\ref{lemma:B:L1} implies that the
 Frobenius norm $\|C_n^{1/2}A_nC_n^{1/2}\|_F$ satisfies the requirements of the
 Theorem~\ref{thm1}. Therefore,
 \[
  \lim_{n \to \infty} (n^{-1} \dual{\xi_n}{A_n\xi_n} - n^{-1}\expectation
  \dual{\xi_n}{A_n\xi_n}) = 0
 \]
 almost surely. We have for every $n$ by Lemma~\ref{lemma:B:L2} that
 \[
 \expectation \dual{\xi_n}{A_n\xi_n} = A_n : C_n = B_n : C_n + \puokki
 (A_n K_n + K_n^* A_n) : C_n.
 \]
 By Propositions~\ref{lemma:B:L3} and~\ref{lemma:B:L4} we have
 \[
 \begin{split}
 |A_n K_n + K_n^* A_n| : |C_n| & \lesssim
 n^{2\alpha_-}\,  \widetilde
 K_n : |C_n| \\
 & \lesssim n^{2(\alpha_- + \beta_+)} \log n.
 \end{split}
 \]
 This implies that
 \[
 \lim_{n \to \infty} n^{-1} \big(\expectation \dual{\xi_n}{A_n\xi_n} -
 B_n : C_n\big) = 0
 \]
 On the other hand
 \[
 n^{-1} B_n : C_n = \torusInt \di t \torusInt \di s
 \frac{g_\beta(t)}{g_\alpha(s)} h_n(t-s)
 \]
 where $h_n$ is the Fejér kernel. Therefore, the limit exists and equals
 the claimed value provided $g_\beta/g_\alpha \in L^1$. This, however,
 is equivalent with $\alpha - \beta > -\puokki$ which follows from the
 assumption.
\end{proof}
The proofs of Propositions~\ref{lemma:B:L1},~\ref{lemma:B:L3}
and~\ref{lemma:B:L4} and proof of Lemma~\ref{lemma:B:L2} are postponed
to Section~\ref{PFB}, as mentioned before. In the end of this section,
we describe what is needed in order to obtain these and introduce two
lemmata that cover the key points.

First, we consider some properties of the matrices $\widetilde
K_n(\alpha)$ in more detail. We note that $K_n(\alpha)$ is a sum of
two products of two Hankel operators corresponding to symbols $g_\alpha$ and
$g^{-1}_\alpha$. If both symbols were bounded, the standard symbol
calculus methods could be used, but in this case the properties has to
be computed by hand.

We will denote
\begin{equation}
  \label{B:D6}
    k_\alpha(x,y) = \frac{(x \vee y)^{-1+2|\alpha|}}{(x \wedge
    y)^{2|\alpha|}} + 
     \frac{(1- x \wedge y)^{-1+2|\alpha|}}{(1-x \vee y)^{2|\alpha|}}.
\end{equation}
\begin{lem}
  \label{lemma:B:L5}
    For $0 < |\alpha| < \puokki$, we have that
    \[
      \widetilde K_n(\alpha)_{ij} \lesssim n^{-1} k_\alpha(x,y)
    \]
    where $n x = (i \vee 1) \wedge (n-1)$ and $n y = (j \vee 1) \wedge (n-1)$.
\end{lem}
Propositions~\ref{lemma:B:L1} and~\ref{lemma:B:L3} require knowledge of the
inverse matrix $T_n^{-1}$. For this we adapt the results of Rambour and
Seghier~\cite{RambourSeghier2, RambourSeghier}. Let us introduce some notations. We will
denote for every $x \in [0,1]$
\begin{equation}
  \label{B:E10}
    \widetilde x = 1 - x \quad \text{and} \quad |x|_n = |x| \vee n^{-1}
\end{equation}
With these notations can define
\begin{equation}
  \label{B:E7}
  \begin{split}
    \mathcal S(f)(x,y) &= 
    f(x,y) \Iverson{y \ge (x \vee \widetilde x)} +
    f(y,x) \Iverson{\widetilde x \le y < x} \\
    &+ f(\widetilde y, \widetilde x) \Iverson{x \le y < \widetilde x}
    + f(\widetilde x, \widetilde y) \Iverson{y < (x \wedge \widetilde
    x)}.
  \end{split}
\end{equation}
Therefore $\mathcal
S(f)$ is an extension of function $f$ defined in a triangle $x \vee
\widetilde x \le y \le 1$ which is symmetric and invariant with respect
to transformation $(x,y) \leftrightarrow (\widetilde x,\widetilde y)$.
We denote
\begin{equation}
  \label{B:E12}
I_d(x) = \brc{(x \vee \widetilde x)\le y < \puokki(x+1)},
\quad
I_b(x) = \brc{y \ge (\widetilde x \vee \puokki(x+1))}
\end{equation}
and furthermore,
\begin{equation}
  \label{B:E8}
    E^{(\alpha,n)}_1(x,y) = |y-x|_n^{-1+2\alpha} \Iverson{
    y \in I_d(x)}
\end{equation}
and
\begin{equation}
  \label{B:E9}
    E^{(\alpha)}_2(x,y) = (y-x)^{-1+\alpha}x^\alpha \widetilde
    y^\alpha \Iverson{y \in I_b(x)}
\end{equation}

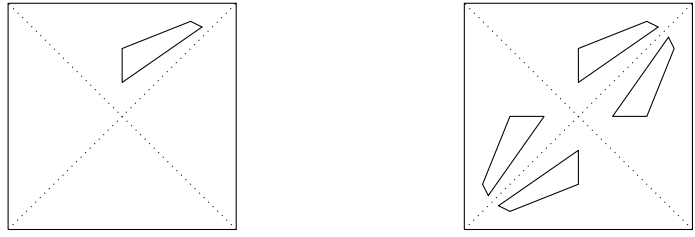
\begin{figure}
\begin{tikzpicture}[scale=3]
  \draw (0,0) -- (1,0) -- (1,1) -- (0,1) -- (0,0);
  \draw  (0.5,0.65) -- (0.85,0.895) -- (0.80,0.92) -- (0.5,0.8) --
  (0.5,0.65);
  \draw  (2.5,0.65) -- (2.85,0.895) -- (2.80,0.92) -- (2.5,0.8) --
  (2.5,0.65);
  \draw  (2.65,0.5) -- (2.895,0.85) -- (2.92,0.80) -- (2.8,0.5) --
  (2.65,0.5);
  \draw  (2.35,0.5) -- (2.105,0.15) -- (2.08,0.20) -- (2.2,0.5) --
  (2.35,0.5);
  \draw  (2.5,0.35) -- (2.15,0.105) -- (2.20,0.08) -- (2.5,0.2) --
  (2.5,0.35);
  \draw[dotted] (0,0) -- (1,1);
  \draw[dotted] (1,0) -- (0,1);
  \draw (2,0) -- (3,0) -- (3,1) -- (2,1) -- (2,0);
  \draw[dotted] (2,0) -- (3,1);
  \draw[dotted] (3,0) -- (2,1);
\end{tikzpicture}
\caption{Illustration of $\mathcal S(f)$}
\end{figure}
These notations allow us to adapt the results for the elementwise
asymptotics for the inverse matrices of Toeplitz matrices
$T_n(g_\alpha)$.
\begin{lem}
  \label{lemma:B:L6}
    When $0 < |\alpha| < \puokki$ and we have
    \[
    \begin{split}
    |T_n(g_{-\alpha})_{ij}^{-1}| \asymp \Iverson{i=j} + \Iverson{i\ne j}n^{-1+2\alpha} \mathcal S(E^{(\alpha,n)}_1
    +E^{(\alpha)}_2)(x,y)
    \end{split}
    \]
    when $xn = i \vee N \wedge (n-N)$, $yn = j \wedge (n-N)$.
\end{lem}
The adaptation is, however, not entrily straightforward so we need
further analysis for obtaining this lemma. This is done in the following
section (Section~\ref{factorisation}).

\section{Factorisation of the symbol $g_\alpha$}
\label{factorisation}
The proof of Lemma~\ref{lemma:B:L6} is mostly technical and it relies
on the asymptotic representation of Rambour and
Seghier~\cite{RambourSeghier2,
RambourSeghier}.
In these articles, they obtain elementwise asymptotic
representations of inverses of Toeplitz matrices with a single
Fisher--Hartwig singularity. More precisely, they give their results to
symbols of form $f_1 \theta_{2\alpha}$ where $f_1$ is sufficiently
smooth positive function (\emph{a smooth perturbation}) and the
$\theta_{2\alpha}$ is the pure Fisher--Hartwig singularity
\begin{equation}
  \label{B:eq:12}
    \theta_{2\alpha}(t) = 2^{\alpha}(1-\cos t)^{\alpha}.
\end{equation}
Since we know already that $g_\alpha \asymp \theta_{-2\alpha}$ we in
principle only have to show that $g_\alpha \theta_{2\alpha}$ is
sufficiently smooth. In~\cite{RambourSeghier2, RambourSeghier}, the
assumption for smooth perturbation $f$ is that $(\widehat f_1(k)
k^{3/2}) \in \ell^1$, which is valid in our case only for $\alpha >
-\nicefrac1{4}$., Therefore in order to handle the case $-\nicefrac1{2} <
\alpha \le -\nicefrac1{4}$ we have to improve their result.

Analysing the proofs of the 
main results in~\cite{RambourSeghier2,
RambourSeghier} we observe that the symbol $g_\alpha$ only needs to satisfies the following
conditions:
\begin{itemize}
  \item[$-$] immediate conditions: $g_\alpha \ge 0$, $g_\alpha \in
  L^1$ and  $g_\alpha^{-1} \in L^1$
  \item[$-$] $\log g_\alpha \in L^1$ (follows from previous, since $|\log
  x| \le x + x^{-1}$).
  \item[$-$] there exists a $q_\alpha \in H^2(\torus)$, a boundary trace of an
  analytic square integrable function, that satisfies $q_\alpha
  \overline{q_\alpha} = g_\alpha^{-1}$ and
  \begin{equation}
    \label{C:eq:1}
    C_\alpha\widehat w_{\alpha}(k) = \widehat q_{\alpha}(k)
    + \littleOh{k^{-\alpha -1}}
  \end{equation}
  where $w_{\alpha}(t) = (1 - e^{it})^{\alpha}$ and $C_\alpha
  = \lim_{t \to 0}q_\alpha w_{-\alpha}(t)$.
\end{itemize}
The last condition means that we factorize the symbol $g_\alpha$ into
the product of an analytic and anti-analytic square root. We have a
trivial factorisation for the pure
Fisher--Hartwig singularity $w_\alpha \overline
w_\alpha = \theta_{2\alpha}$. The condition then states that the
Fourier coefficients of the analytic square root coincide with the
Fourier coefficients of the analytic square root of the pure
Fisher--Hartwig singularity asymptotically and upto a constant.

Following Grenander--Szeg\H{o}~\cite{GrenanderSzego} we have a clear
recipe for this factorisation (of $f \ge 0$ defined on $\partial\mathbb D \sim
\torus$, say)
\begin{itemize}
  \item[$-$] let $u$ be the harmonic extension of $\nicefrac1{2}\log f$ in $\mathbb
  D$
  \item[$-$] let $v$ be the harmonic conjugate of $u$ with $v(0) = 0$
  \item[$-$] the required analytic square root of $f$ is then the
  boundary trace of $\exp(F)$, where $F = u + iv$.
\end{itemize}
This \emph{Riemann--Hilbert problem} has a unique solution if $f \in
L^1$ and $\log f \in L^1$.
\begin{lem}
  \label{lemma:C:L1}
    When $f \in L^1(\torus)$, $f \ge 0$ and $\log f \in L^1(\torus)$,
    then the function $q$ given by
    \begin{equation}
      \label{C:eq:2}
	q = \sqrt f \exp\big(i/2 \hilbert(\log f)\big)
    \end{equation}
    satisfies
    \begin{equation}
      \label{C:eq:2b}
	q \in H^2(\torus) \quad \text{and} \quad f = q \overline q
    \end{equation}
    where $\hilbert$ is the Hilbert transform on the torus
    \[
    \hilbert f (x) = \text{p.v.}\,\torusInt \cot \Big(\frac{t-x}{2}\Big) f(t) \di
    t
    \]
    Moreover, a function $q$ satisfying~\eqref{C:eq:2b} is unique upto a multiplication with an inner function.
\end{lem}
In the sequel we will denote 
\begin{equation}
  \label{C:eq:3}
    \psi_\alpha^{-2} = g_\alpha \theta_{2\alpha}
\end{equation}
the perturbation of the pure Fisher--Hartwig symbol that we need to
obtain the FBN symbol $g_\alpha$. We will denote by
\begin{equation}
  \label{C:eq:4}
    r_\alpha = \psi_\alpha \exp\big(i \hilbert (\log\psi_\alpha)\big)
\end{equation}
the analytic square root of $\psi_\alpha^2$ given by
Lemma~\ref{lemma:C:L1}.
We have
\begin{lem}
  \label{lemma:C:L2}
    For every $0 < |\alpha| < \nicefrac1{2}$ there exists an $q_\alpha \in
    H^2(\torus)$ such that $q_\alpha \overline{q_\alpha} = 1/g_\alpha$
    and for almost every $t \in \torus$ it holds that
    \[
    q_\alpha(t) = w_\alpha(t) r_\alpha(t).
    \]
\end{lem}
This representation is enough for showing the required estimate for the
Fourier coefficients.
\begin{lem}
  \label{lemma:C:L3}
    For every $0 < |\alpha| < \nicefrac1{2}$ and for every $k \ge 1$ we have
    \[
    \widehat q_\alpha(k) = C_\alpha \widehat w_\alpha(k) +
    \bigOh{k^{-2-\alpha}}.
    \]
\end{lem}
\begin{proof}[Proof of Lemma~\ref{lemma:C:L3}]
The result follows by combining Lemmata~\ref{lemma:C:L2}
and~\ref{lemma:C:L6} since the Fourier transform of the product is the
convolution of the Fourier transforms.
\end{proof}
This result implies that we may use the elementwise results of
Rambour and Seghier since $g_\alpha$ satisfies the
condition~\eqref{C:eq:1}, even though the perturbation $\psi_\alpha^2$
is not always as smooth as they required (see Lemma~\ref{lemma:C:L4}).

The estimate for the convolution of Fourier transforms
(Lemma~\ref{lemma:C:L6}) follows from the following lemmata.
\begin{lem}
  \label{lemma:C:L4}
    Let $u = \log\psi_\alpha$. Then
    we have the following asymptotic estimates for the Fourier
    coefficients
    \begin{equation}
      \label{C:eq:5}
      \abs{\widehat r_\alpha(k)} = \abs{\widehat u(k)} \asymp
      \abs{\widehat \psi_\alpha(k)} 
      \asymp k^{-3-2\alpha}
    \end{equation}
\end{lem}
\begin{lem}
  \label{lemma:C:L5}
    When $\alpha \ne 0$ we have that
    \[
    \widehat w_\alpha(k) = \Gamma(-(1+\alpha))^{-1} k^{-(1+\alpha)} \big(1 + \bigOh
    {k^{-1}}\big)
    \]
\end{lem}
Combining these two lemmata we obtain the asymptotic representation for
the Fourier coefficients of $w_\alpha r_\alpha$ namely
\begin{lem}
  \label{lemma:C:L6}
    We have
    \[
    \widehat w_\alpha * \widehat r_\alpha (k) = \widehat w_\alpha (k)
    r_\alpha(0) + \bigOh {k^{-2-\alpha}}
    \]
\end{lem}
\section{Uniform Law of Large Numbers and estimates}
\label{USLLN}
We want to show the uniform strong law of large numbers (uniform SLLN)
that we will use to obtain estimate for the posteriori distribution.
\begin{thm}[uniform SLLN]
  \label{thm3}
  Let $I$ denote the interval $(\wh H - \nicefrac1{2} ,1) \cap (0,1)$.
  Then
  \[
  \mathbb P(\rho_n \to \rho_\infty \text { uniformly on compact subsets
  of }I) = 1
  \]
  where $\rho_n = \alpha \mapsto Q_n(\xi_n, \alpha)$ and $\rho_\infty =
  \alpha \mapsto F(\wh H, \alpha)$.
\end{thm}
This follows from the pointwise strong law of large
numbers (Theorem~\ref{thm2}).  The extension to uniform convergence is done with the help of
bounded variation with respect of the parameter $\alpha$. This in effect
can be reduced to monotonicity for auxiliary functions. Therefore, we
consider how the family of derivatives $(\rho_n')_n$ behaves.
\begin{lem}
  \label{lemma:ext:l1}
    For every $\alpha$ we have
    \[
    \partial_\alpha T_n(g_\alpha)^{-1}
    = -T_n(g_\alpha)^{-1} \big(\partial_\alpha T_n(g_\alpha)\big)
    T_n(g_\alpha)^{-1}.
    \]
\end{lem}
We first consider the part when $\alpha > 0$ (in our case $\alpha = 0$
corresponds to identity matrices so the case is trivial).
\begin{lem}
  \label{lemma:ext:lu1}
    There exists an $\lambda_1 > 0$ and an $\lambda_2 > 0$ such that 
    for every $n$ the following estimates hold:
    \begin{enumerate}
      \item 
      $\forall \alpha \in (0,\puokki)\colon T_n(g_\alpha)^{-1} \le \lambda_1 I_n$
      \item 
      $\forall \alpha \in (0, \puokki)\colon \partial_\alpha T_n(g_\alpha) \ge -\lambda_2 I_n$, 
      \item
      $\forall \alpha \in (0, \puokki)\colon \partial_\alpha T_n(g_\alpha)^{-1} \le \lambda_1^2 \lambda_2 I_n$
    \end{enumerate}
\end{lem}
The part when $\alpha < 0$ is similar, but in this case the zero in the
symbol causes the sequence $(T_n(g_\alpha)^{-1})_n$ become unbounded and
to handle that we need to estimate the matrices with unbounded symbols.
\begin{lem}
  \label{lemma:ext:lu2}
    For every $\gamma \in (-\puokki,0)$ there exists a
    $\lambda_3 > 0$ and a $\lambda_4 > 0$ such that for 
    for every $n$ the following estimates hold:
    \begin{enumerate}
      \item 
      $\forall \alpha \in [\gamma,0)\colon T_n(g_\alpha)^{-1} \le \lambda_3 T_n(g_\gamma)^{-1}$
      \item 
      $\forall \alpha \in (-\puokki,0)\colon \partial_\alpha
      T_n(g_\alpha) \ge -\lambda_4 T_n(g_\alpha)$
      \item
      $\forall \alpha \in [\gamma,0)\colon \partial_\alpha T_n(g_\alpha)^{-1} \le 
      \lambda_3\lambda_4 T_n(g_\gamma)^{-1}$.
    \end{enumerate}
\end{lem}
With these estimates we can show the equicontinuity of the family
$\{\rho_n\}$. First we construct auxiliary increasing family of
functions.
\begin{lem}
  \label{lemma:ext:lu3}
    Suppose that for some fixed sequence $(z_n)$, for some $\gamma$ as in
    Lemma~\ref{lemma:ext:lu2} and for some $\alpha_+  \in (\gamma,
    \nicefrac1{2})$, the function
    \[
    m(\alpha) := \sup_n Q_n(z_n, \alpha)
    \]
    is finite for $\alpha = \gamma$ and $\alpha = \alpha_+$. 
    Suppose further that $0 < c \le \norm {z_n} \le C$
    for all $n$. Let $M =  (\lambda_1^2\lambda_2) \vee (\lambda_3 \lambda_4 m(\gamma)/c^2)$.
    Then auxiliary functions $\widetilde Q_n(\alpha) := m(\gamma) - Q_n(z_n, \alpha)
    + M \alpha \norm {z_n}^2
    $ have the following properties:
    \begin{enumerate}
      \item for every $n$ the function $\widetilde Q_n$ is increasing
      and continuous on $[\gamma,\alpha_+]$.
      \item the functions $\widetilde Q_n$ are uniformly bounded from below with lower
      bound $0$
      \item the functions $\widetilde Q_n$ are uniformly bounded from
      above with upper bound $m(\alpha_+)+M\alpha_+C^2$.
    \end{enumerate}
\end{lem}
\begin{proof}
    This follows immediately from Lemmata~\ref{lemma:ext:lu1}
    and~\ref{lemma:ext:lu2} since for $\alpha > 0$ we know
    \[
    \partial_\alpha \widetilde Q_n(\alpha) \ge -\lambda_1^2\lambda_2 \norm {z_n}^2 + M \norm {z_n}^2
    \ge 0
    \]
    and when $\alpha \in [\gamma,0)$ we have
    \[
    \partial_\alpha \widetilde Q_n(\alpha) \ge -\lambda_3\lambda_4 m(\gamma) + M \norm {z_n}^2
    \ge c^2(-\lambda_3\lambda_4 m(\gamma)/c^2 + M) \ge 0.
    \]
\end{proof}
We can now prove the uniform convergence for the auxiliary functions by
Helly's Selection Theorem.
\begin{lem}
  \label{lemma:ext:lu4}
    Suppose for fixed $z = (z_n)$ we know that on a dense subset $J$ of
    the interval $[\gamma, \alpha_+]$ 
    \[
    \forall \alpha \in J\colon \lim_{n\to\infty} \widetilde \rho_n(\alpha) =
    \rho_\infty(\alpha)
    \]
    where $\widetilde \rho_n(\alpha) = Q_n(z_n, \alpha)$. If in addition
    $\{\alpha_+, 0,\gamma\} \subset J$, then the sequence $(\widetilde
    \rho_n)$ converges to $\rho_\infty$ uniformly on $[\gamma,
    \alpha_+]$.
\end{lem}
\begin{proof}
  Since $\rho_n(0) = \norm {z_n}^2$ converges to $\rho_\infty(0) \in
  (0,\infty)$,
  we can without a loss of generality assume that $0 < c \le \norm{z_n} \le C <
  \infty$ for all
  $n$ since the condition could be violated only finitely many times and
  we could replace $\widetilde\rho_n$ by $\widetilde\rho_{n+N}$. 

  Since $\gamma$ and $\alpha_+$ are in $J$ we may without a loss of
  generality
  assume that every $f \in \{\widetilde \rho_n\}\cup \{\rho_\infty\}$
  the function $f$ is increasing and uniformly bounded from above and
  from below. This follows by considering functions $f + \kappa$
  instead, where $\kappa(\alpha) = M\alpha$.
  
  The functions $\widetilde\rho_n + \kappa$ are increasing and uniformly
  bounded from above and below by Lemma~\ref{lemma:ext:lu3} for large
  enough $M > 0$.  Moreover, by taking the $M$ even larger, if
  necessary, we can assume the same for the function $\rho_\infty +
  \kappa$. Furthermore, it is enough to show the uniform convergence for
  these functions.

  So let us suppose that all the functions are increasing, continuous and
  uniformly bounded from above and below. Choose any subsequence
  $(\phi_n) \subset (\widetilde\rho_n)$. By Helly's Selection Theorem
  the sequence $(\phi_n)$ has a subsequence $(\widetilde\phi_n)$ that
  converge pointwise to an increasing function $\phi_\infty$ on
  $[\gamma, \alpha_+]$. Moreover, the convergence is uniform if the
  function $\phi_\infty$ is continuous. Since $\widetilde\phi_n$
  converges to $\rho_\infty$ on $J$, we see that $\phi_\infty(\alpha) =
  \rho_\infty(\alpha)$ for every $\alpha \in J$. For every point of 
  continuity $\beta$ of $\phi_\infty$, we know that
  \[
  \phi_\infty(\beta) = \sup_{\alpha < \beta, \alpha \in J}
  \rho_\infty(\alpha) = \rho_\infty(\beta).
  \]
  If $\beta$ would be a point of discontinuity, we would know that
  \[
  \sup_{\alpha < \beta, \alpha \in J}\phi_\infty(\alpha) 
  <
  \inf_{\alpha > \beta, \alpha \in J}\phi_\infty(\alpha) 
  \]
  but since both left and right hand sides are $\rho_\infty(\beta)$ by
  continuity of $\rho_\infty$, we have a contradiction. Therefore, we
  may deduce that $\phi_\infty =
  \rho_\infty$ and thus $(\widetilde \phi_n)$ converges uniformly to
  $\rho_\infty$ on $[\gamma, \alpha_+]$. This in turn implies that
  $(\widetilde\rho_n)$ converges uniformly to $\rho_\infty$ on $[\gamma,
  \alpha_+]$ since the uniform convergence is topological convergence.
\end{proof}
With these lemmata, we obtain the Theorem~\ref{thm3} in a straight
forward manner.
\begin{proof}[Proof of Theorem~\ref{thm3}]
Let $I' = I - \nicefrac1{2} = (\beta_-,\beta_+)$ where $I$ is as stated in the claim. 
Let $k > 0$ and choose $\alpha_{-,k}, \alpha_{+,k} \in \mathbb Q$ such
that $\beta_- < \alpha_{-,k} < \beta_- + \frac 1k < \beta_+ - \frac 1k <
\alpha_{+,k} < \beta_+$.

Choose a countable dense set $J = [\alpha_{-,k},\alpha_{+,k}] \cap
\mathbb Q$. Theorem~\ref{thm2} implies that
\[
\widetilde \Omega = \{\forall \alpha \in J\colon
\lim_n \rho_n(\alpha) = \rho_\infty(\alpha) \}
\]
is an almost sure event. Let $z_n = \xi_n(\omega)$ for $\omega \in
\widetilde \Omega$. Application of Lemma~\ref{lemma:ext:lu4} implies
that $\widetilde \rho_n := \rho_n(\omega)$ converges uniformly to
$\rho_\infty$ on $I' \cap [\gamma, \alpha_+]$. Therefore, we deduce that
\[
\widetilde \Omega \subset \Omega_k := \{ \rho_n \to \rho_\infty
\text{ uniformly on } [\alpha_{-,k},\alpha_{+,k}] \}
\]
This implies that $\mathbb P (\bigcap_k \Omega_k) = 1$ and the claim
follows.
\end{proof}

Since the convergence takes place not in the whole interval $(-\nicefrac1{2},\nicefrac1{2})$ we
need some estimates to handle the remaining parts so that we can at
least obtain the parameter estimation result. For this we would need to
obtain an upper bound for the quadratic forms $\rho_n(\alpha)$ for all
$\alpha$. When $\alpha > 0$ we have an upper bound by
Lemma~\ref{lemma:ext:lu1}, but for $\alpha < 0$ the zero in the symbol
causes the singularity that caused an restriction.

\begin{lem}
  \label{lemma:est:lu1}
    For every $\e > 0$ there exists a constant 
    $\lambda_5 > 0$ such that
    every $\alpha \in [-\puo+\e,-\e]$ there exists a symbol $\widetilde g_\alpha$ 
    \begin{enumerate}
    \item $\widetilde g_\alpha \le g_\alpha \le \lambda_5 \widetilde g_\alpha$ and
    \item $\partial_\alpha \widetilde g_\alpha \ge 0$
    \item $\lambda_5^{-1}T_n(\widetilde g_\alpha)^{-1} \le  
    T_n(g_\alpha)^{-1} \le  
    T_n(\widetilde g_\alpha)^{-1}$
    \end{enumerate}
\end{lem}
The existence of these auxiliary symbols $\widetilde g_\alpha$ imply
that $T_n(g_\alpha)^{-1} \ge \lambda_5^{-1} T_n(g_\beta)^{-1}$ for every
$\alpha < \beta < 0$ and 
therefore, we obtain
\begin{thm}
  \label{theorem102}
  Let $I$, 
  $\rho_n$ and $\rho_\infty$ be as in Theorem~\ref{thm3}.
  Then
  \[
  \begin{split}
  \mathbb P(\rho_n \to \rho_\infty \  &\text{ uniformly on compact subsets
  of }I \text{ and } \\
  &\lim_n \inf_{t \notin I} \rho_n(t) = \infty) = 1. 
  \end{split}
  \]
\end{thm}

\section{Parameter estimation from the posterior}
\label{param-estimation}
According to Theorem \ref{thm3} there exists an almost sure
event $\Omega' \subset \Omega$ such
that $\forall \omega \in \Omega', \forall \wh H \in (0,1), \forall  \alpha \in (0,1)$ if $\alpha > \wh H-\puo$
then
\begin{equation}
  \label{eq49}
\lim_{n \to \infty} 
\frac
    {n^{2 (\alpha-\wh H)} Q_n(z_n,\alpha)}
    {n^{2 (\alpha-\wh H)+1} F(\alpha,\wh H)}
= 1.
\end{equation}
where $z_n := \xi_n(\wh H)(\omega )$ and $Q_n(z_n,\alpha) :=
{\dual{z_n}{{\inverse{{T_n({f_\alpha})}}}z_n}}$. The
deterministic function $F$ is the expected value
\[
F(\alpha, \beta) := \torusInt \frac{f_\beta(t)}{f_\alpha(t)} \di t .
\]
\begin{prop}
  \label{lause10}
    We have $\forall \alpha \in (0,1), \e > 0$ the determinant has the asymptotic
    estimate
    \begin{equation}
      \label{eq48}
	\abs{\frac{\lvert T_n(f_\alpha) \rvert}{G(\alpha)^n (1+n)^{(1-2\alpha)^2/4}} - E(H)} \leq
	\frac{C(\alpha)}{(1 + n)^{1-\e}},
    \end{equation}
    for large enough $n$ where \(C,G, E\colon (0,1) \to \R^+ \setminus \brc{0}\) are continuous
    functions with respect to $H$.
\end{prop}
\begin{proof}
  The claim follows directly from \cite[Theorem 2.5]{ES:97}, since in our case the
  symbol $f_H$ has exactly one Fisher--Hartwig singularity and by
  Lemma~\ref{lemma:C:L4}, the
  assumptions of
  \cite[Theorem 2.5]{ES:97} are fulfilled with parameters $\delta = \gamma =
  \alpha$.
  This implies the claim.
\end{proof}
Without a loss of generality, we may assume that the set $\wh B$ is an
interval $[0,
\beta)$ for some $\beta < 1$.  According to Proposition \ref{lause10} we have
$\forall \alpha \in
(0,1)$ that
\begin{equation}
  \label{eq50}
  \log {|T_n(f_\alpha)|} = \bigOh n
\end{equation}
Therefore, by Dominated Convergence 
\begin{equation}
  \label{eq51a}
  \begin{split}
    &\lim_{n \to \infty} C_n(z_n) \int_{\alpha_-}^\beta n^{n\alpha}
	|T_n(f_\alpha)|^{-\nicefrac1{2}} \exp\big(-\puokki n^{2 (\alpha-\wh H)}
	Q_n(z_n,\alpha)
	\big) \di \alpha \\
    &= \lim_{n \to \infty} C_n(z_n) \int_{\alpha_-}^\beta 
    \exp\big(\alpha n \ln n - \puokki{n^{2(\alpha - \wh{H})+ 1}}
    F(\alpha,
    \wh{H}) ) \di \alpha,
  \end{split}
\end{equation}
where $\alpha_- = (\wh H - \nicefrac1{2})_+$.
Similarly, we can estimate
\begin{equation}
  \label{eq51}
  \begin{split}
    &\limsup_{n \to \infty} C_n(z_n) \int_0^{\alpha_-} n^{n\alpha}
	|T_n(f_\alpha)|^{-1/2} \exp(-\puokki n^{2 (\alpha- \wh H)}
	Q_n(z_n,\alpha) ) \di \alpha \\
    &\le \limsup_{n \to \infty} C_n(z_n)\int_0^{\alpha_-} 
    \exp(\alpha n \ln n) \di \alpha.
  \end{split}
\end{equation}
We know that the function $\alpha \mapsto F(\alpha, \wh H)$ is continuous and
$F(\alpha,\alpha) = 1$. If we replace the function $\wh F := \alpha
\mapsto F(\alpha,\wh H)$ with a
constant function ${\mathbf 1} := \alpha \mapsto 1$ the function 
\[
K_n(L)(\alpha) :=  \alpha n \log n - \puokki n^{2(\alpha - \wh H)+1}
L(\alpha)
\]
inside the exponent function on the right-hand side of the
equation~\eqref{eq51a} would become
\[
k_n(\alpha) := K_n({\mathbf 1})(\alpha) = \alpha n \log n - \puokki
n^{2(\alpha - \wh H)+1}.
\]
Differentiation with respect to $\alpha$ reveals that
\[
k'_n (\alpha) = n\log n \bpa{1  - n^{2(\alpha - \wh H)} }
\]
which is negative when $\alpha > \wh H$ and positive when $\alpha < \wh H$. It has
a unique zero at $\alpha = \wh H$ which means that the function $k_n$ has a
unique global maximum at $\alpha = \wh H$. 

Since $k_n = K_n({\mathbf 1})$ has a unique global maximum at $\alpha = \wh
H$ and the functions $\wh F$ and $\wh F'$ are continuous, we could
expect that the function $K_t(\wh F)$ would also have a global maximum
\emph{near} the point $\alpha = \wh H$.
In order to show that this is, indeed, the case we differentiate $\kappa_n := K_n(\wh F)$ which gives
that
\begin{lem}
  \label{monoton}
  We have 
  \[
  \kappa_n'(\alpha) = n \log n \bpa{1 - n^{2(\alpha - \wh H)} \fii_n(\alpha)}
  \]
  and
  \[
  \kappa_n''(\alpha) = -n (\log n)^2 n^{2(\alpha - \wh H)} \psi_n(\alpha)
  \]
  where
  \[
  \fii_n(\alpha) := \wh F(\alpha) + \frac{\wh F'(\alpha)}{2 \log n}
  \]
  and
  \[
  \psi_n(\alpha) := 2 \fii_n(\alpha) + \frac{\wh F'(\alpha)}{\log n}+
  \frac{\wh F''(\alpha)}{2(\log n)^2}
  \]
\end{lem}
In the sequel, we use maximum and minimum operators defined as
\[
M(f) := \sup \set{f(x)}{x \in (\wh H-\nicefrac1{2} + \e, 1)}
\]
and
\[
m(f) := \inf \set{f(x)}{x \in (\wh H-\nicefrac1{2} + \e, 1)}.
\]
We have to cut out a small neighbourhood of $\wh H-\nicefrac1{2}$ since the function
$\wh F$ explodes at $\wh H -\nicefrac1{2}$.
\begin{lem}
  \label{estimates}
  There is $N \in \N$ such that $\forall n \ge N$ we have
  \[
  \puokki m(\wh F) \le \fii_n \le 2 M(\wh F)
  \quad
  \text{and}
  \quad
  \psi_n \ge \puokki m(\wh F).
  \]
\end{lem}
\begin{lem}
\label{concave}
  The function function $\kappa_n$ is concave for every $n$ large enough.
\end{lem}
\begin{proof}
This follows from Lemmata~\ref{monoton} and \ref{estimates}.
\end{proof}
\begin{lem}
  \label{zeropoint}
    Let $\alpha_+(n) := \wh H + (\log n)^{-1} (\log 2 - \log m(\wh F))$ and 
    let $\alpha_-(n) := \wh H - (\log n)^{-1} (\log 2 + \log M(\wh F))$. 
    When $N$ is defined as in Lemma~\ref{estimates} then $\forall n \ge N$ the
    equation $\kappa_n'(\alpha) = 0$ has a unique solution $\alpha_n$ inside the interval
    $[\alpha_-(n), \alpha_+(n)]$.
\end{lem}
\begin{proof}
Suppose $\alpha > \alpha_+(n)$. Then by Lemma~\ref{estimates} we have
\[
\kappa_n'(\alpha) < n \log n \bpa{1 - \puokki m(\wh F) n^{2(\alpha_+(n)-\wh H)}}.
\]
Since
\[
n^{2(\alpha_+(n)-\wh H)} = \exp\bpa{\log (2/m(\wh F))} = 2/m(\wh F)
\]
we see that $\kappa_n'(\alpha) < 0$. In the same way, if we suppose $\alpha <
\alpha_-(n)$,
we similarly
\[
\kappa_n'(\alpha) > n \log n \bpa{1 - 2M(\wh F) n^{2(\alpha_-(n)-\wh H)}} \ge 0.
\]
Therefore, for every $n \ge N$ the continuous and decreasing function
$\kappa_n'$ changes sign on interval $[\alpha_-(n), \alpha_+(n)]$. This gives the claim.
\end{proof}
We can asymptotically solve the equation $\kappa_n'(\alpha_n) = 0$. 
\begin{lem}
\label{mean}
  We have that 
  \[
  \alpha_n = \wh H - \frac {\wh F'(\wh H)}{4 (\log n)^2} + \bigOh (\log n)^{-3}
  \]
\end{lem}
\begin{remark}
  The numeric computations and more qualitative arguments indicate that
  $\wh F'(\wh H) > c > 0$ for every $\wh H$. Therefore, the maximum
  aposteriori estimate is biased to left of the true value. Furthermore, we
  see that as $n$ grows to infinity the maximum aposteriori estimate
  becomes asymptotically unbiased.
\end{remark}
Since we need the values of the higher derivatives of $\kappa_n$ at
$\alpha_n$, let's compute them.
\begin{lem}
\label{lemma77}
  We have that 
  \[
  \kappa_n''(\alpha_n) = -2n (\log n)^2 c_n
  \]
  and
  \[
  \kappa_n^{(3)}(\alpha_n) = -4n (\log n)^2 s_n
  \]
  where $c_n = 1 + \bigOh (\log n)^{-1}$ and $s_n = 1 + \bigOh{(\log
  n)^{-1}}$.
  Furthermore, we can estimate that 
  \[
  \forall \alpha \in U_n(\alpha_n)\colon | \kappa_n^{(4)}(\alpha) | \lesssim n (\log n)^4
  \]
  where $U_n(\alpha_n) = \{\,  |\alpha - \alpha_n| \ll (\log n)^{-1} \}$.
\end{lem}

When $\alpha < \wh H-\puokki + \e$ we cannot use the derivatives to study the
extremal points but we have a simple estimate for the function $\kappa$
itself. Since $\wh F$ is positive function, for every $\alpha \in (\wh
H-\puokki, \wh H-\puokki+\e)$ we have that
\[
\kappa_n(\alpha) \le \alpha n \log n \le
(\wh H- \puokki+\e) t \log t.
\]
The same estimate holds on the interval $[0,\wh H-\puokki]$ as well.
Since
\[
\kappa_n(\alpha_n) \ge \kappa_n(\wh H) = \bpa{\wh H  \log n - \puokki} n >
(\wh H-\puokki+\e)n \log n
\]
that holds when $\e < \frac 14$ and $\log n > 2$, we infer that $\kappa_n$ has
its global maximum at $\alpha_n$.

This leads to the Laplace method type argument, since we rescale the
maximum to be one. In other words, we define
\[
I_n(\wh V) := e^{-K_n(\wh F)(\alpha(n))}\int_{\wh V} e^{K_n(\wh
F)(\alpha)} \di \alpha 
\]
and for the upper bound of the remainder part
\[
J(n) := \Iverson{\wh H > \puokki} e^{-K_n(\wh F)(\alpha(n))}\int_0^{\wh
H - \puokki} e^{\alpha n \log n} \di \alpha 
\]
Following the usual procedure, we divide the integration interval into
tail parts and the main part. In the remaining part of
this section we will denote the the left tail interval as
$(0,\beta_-(n))$ and the right tail interval $(\beta_+(n), 1)$. The main part is the
interval between $\beta_-(n)$ and $\beta_+(n)$.

The next lemma shows that we can express the upper bounds
$\rho_{\pm}(n)$ of the tail errors in terms of the derivatives of
$\kappa_n$.
\begin{lem}
\label{tail_estimates}
  For fixed $n$ we have for every $\beta_-(n) < \alpha(n) < \beta_+(n)$ that
  \[
  I_n((0,\beta_-(n)) = \bigOh \rho_-(n)
  \]
  and
  \[
  I_n((\beta_+(n),1) = \bigOh \rho_+(n)
  \]
  where $\rho_-(n) := 1/ \kappa_n'(\beta_-(n))$ and $\rho_+(n) := -1/ \kappa_n'(\beta_+(n))$.
\end{lem}
The next lemma shows that we can express the upper bounds
$\rho_{\pm}(n)$ of the tail errors in terms of
the distance of $\beta_\pm(n)$ from the zero point $\alpha_n$.
\begin{lem}
 \label{tail_error}
  When the distance of $\beta_\pm(n)$ from $\alpha_n$ is $\e_n n^{-\nicefrac1{2}}(\log n)^{-1}$ with $\e_n
  \ll n^{\nicefrac1{2}}$ the upper bounds $\rho_\pm(n)$ of tail estimates are of order 
  \[
  \rho_\pm(n) \asymp \frac 1 {\e_n n^{\nicefrac1{2}} \log n}
  \]
\end{lem}

The only remaining interval is $V = [\beta_-(n), \beta)$ for some $\beta <
\beta_+(n)$. The function
\[
k_n(\alpha) := K_n(\wh F)(\alpha+\alpha(n)) - K_n(\wh F)(\alpha(n))
\]
has a zero at $\alpha = 0$. Since $\alpha(n)$ is the global maximum, we know
that $k_n'(0)=0$ and $k_n''(0) = -nc_n(\log n)^2 <0$. This leads to the
following result.
\begin{lem}
\label{main}
  Suppose $|\beta_\pm - \alpha_n| = \e_n n^{-\nicefrac1{2}}(\log n)^{-1}$ for some $\e_n \ll n^{\nicefrac1{2}}$.
  Let $\tau$ be a change of variable $\tau_n(\beta) = (\beta-\alpha_n) \log n \sqrt {n
  c_n}$ and let $\e = n^{-1/2} \lambda_n^{-1}$ for some $\lambda_n = \littleOh 1$.
  Then
  \[
  \frac 1 {\sqrt{2\pi}}I_n([\beta_-(n), \beta)) = 
  \frac 1{\log n{\sqrt{c_n n}}}
  \Bpa{\varPhi \circ \tau(\beta) - \lambda_-(n) } + \bigOh{ \frac{\e_n} {n\log n}}.
  \]
  where $\lambda_-(n) = {\varPhi \circ \tau(\beta_-(n)}$.
\end{lem}
\begin{lem}
  The optimal choice for the $\e_n$ in order to minimize the error
  estimate is $\e_n \asymp n^{1/4}$. 
\end{lem}
\begin{proof}
  According to Lemma~\ref{main} the error term is increasing in $\e_n$.
  The errors coming from the tails are decreasing in $\e_n$ according to
  Lemma~\ref{tail_error}. Either we get the minimum at the end points
  or at the point of crossing.

  When $\e_n$ is nearly $1$ then the error from the main part is almost
  of order $n^{-1}(\log n)^{-1}$ and the tail error is almost of order
  $n^{-\nicefrac1{2}}(\log n)^{-1}$. 
  
  When $\e_n$ is just under the upper bound $n^{1/2}$ the error from
  tails is almost than $n^{-1}(\log n)^{-1}$ but the
  error from the main part is only of order $n^{-\nicefrac1{2}} (\log n)^{-1}$.

  Therefore, the minimum is obtained at the point of crossing. This
  happens when
  \[
  \frac {\e_n}{n\log n} = \frac 1 {\e_n n^{\nicefrac1{2}} \log n}
  \]
  and the claim follows.
\end{proof}
With these choices we have found that
\begin{lem}
We have
\[
\conditionalProbability{H \le t}{\vec Y_n(\omega )} 
= \varPhi(t_n)\bpa{1+ \bigOh {n^{-1/4}}}
\]
where $t_n = (t-\alpha_n) \sqrt{c_n n}\log n$. 
\end{lem}
Moreover, this implies that for every $\wh B = (0,t)$ we have that
\[
\begin{split}
\conditionalProbability{H \le t}{\vec Y_n(\omega )} 
= \varPhi(t_n)\bpa{1+ \bigOh {n^{-1/4}}}
\end{split}
\]
where $t_n = (t-\alpha_n) \sqrt{c_nn}\log n$. In order to make both sides
coincide better we denote $\wt H_n := (H-\alpha_n)\sqrt{c_nn\log n}$.
Then
\[
\brc{H \le t} = 
\brc{H-\alpha_n \le t - \alpha_n} = 
\{\wt H_n \le t_n\}.
\]
As the final conclusion we get
\begin{thm}
  \label{lause11}
    There exists an $\alpha_n \in (0,1)$, a bounded sequence $c_n$ and $M >
    0$ such that $\abs{\alpha_n - \wh H} \le M/(\log n)^2$. Furthermore, there
    exists an almost sure event $\Omega' \subset \Omega$ so that for every $\omega \in \Omega'$ 
    the conditional distribution of
    \[
    \conditionalProbability{\wt H_n \le t}{\vec Y_n(\omega )} 
    \]
    is asymptotically standard normal distribution $\varPhi(t)$ when
    \[
    \wt H_n := (H- \alpha_n)\sqrt{c_nn}\log n.
    \]
\end{thm}
\begin{cor}
  \label{seuraus12}
    The conditional mean and the maximum aposteriori estimates
    of $H$ are both equal to $\alpha_n$ asymptotically. Moreover, both
    are asymptotically unbiased estimators of $\wh H$. The conditional
    variance has a formula
    \[
    \expectation\bpa{{(\wh H - \alpha_n)^2} \ehto {\vec Y_n}} = \frac
    1{c_1(n) n (\log n)^2}(1 + \bigOh {n^{-\nicefrac1{4}}}).
    \]
\end{cor}
\begin{remark}
  Since the posteriori variance converges faster to zero than the
  expectation, we note that for large $n$ the posteriori solution would
  falsely give confidence intervals that will not intersect with the
  true value. However, first few digits would be still reliable.
\end{remark}
\section*{Acknowledgments}

LP and PP have been partially funded by European Research Council (ERC Advanced Grant
267700 - InvProb, PI Lassi P\"aiv\"arinta).  The research of LP was
supported by Estonian government grant PUT1093.  The work of PP has in
addition been funded by Academy of Finland (decision
numbers 250215 and 284715 Centre of Excellence in Inverse Problems
Research 2012-2017) and
Finnish Funding Agency for Technology and Innovation (project 40370/08).
He would also like to thank Mikko Kaasalainen and the Department of
Mathematics at the Tampere University of Technology, where part of the
work was carried out.
\appendix
\section{Proofs of auxiliary results in Section~\ref{SLLNGQF}}
\label{PFA}
In this section we prove the technical results that were mentioned in
Section~\ref{SLLNGQF}.
For the proof of Lemma~\ref{lemma:L8} we need a following result.
\begin{lem}
  \label{lemma:L17}
    We have for $N \ge 1$ and any matrix $B$ that
    \[
    \Theta(B,N) = (B^s : C)\Theta(N-1) + 2(N-1) \Theta(B^sCA, N-1).
    \]
    where $B^s = \puokki(B+B^\top)$.
\end{lem}
With the help of this auxiliary result we can prove
Lemma~\ref{lemma:L8}.
\begin{proof}[Proof of Lemma~\ref{lemma:L8}]
We will show that
\begin{equation}
  \label{eq:L8:1}
  \Theta(N) = \sum_{j = 1}^k \beta_j 
  + 2^k (N-1)^{\underline k} \Theta((AC)^kA,N-k)
\end{equation}
for every $k = 1, \dots, N - 1$ where
\[
\beta_j := 2^{j-1}(N-1)^{\underline{j-1}} R_j \Theta(N-j).
\]
The proof is by induction with respect to $k$. When $k = 1$, this is the
special case of Lemma~\ref{lemma:L17} for $B = A$.

Assuming that identity~\eqref{eq:L8:1} holds for $k < N-1$, then
\[
\Theta(N) = \sum_{j=1}^k \beta_j + 2^k(N-1)^{\underline k}
\Theta((AC)^kA, N-k).
\]
Since by Lemma~\ref{lemma:L17}
\[
\begin{split}
\Theta((AC)^kA, N-k) & = R_{k+1} \Theta(N-(k+1)) \\
&+ 2(N-(k+1))
\Theta((AC)^{k+1}A, N-(k+1))
\end{split}
\]
the identity~\eqref{eq:L8:1} holds for $k+1 \le N-1$, as well, and
therefore, by induction for every $k = 1, \dots, N-1$. When $k = N-1$,
the last term on the right-hand side of the identity~\eqref{eq:L8:1}
reduces to
\[
2^{N-1} (N-1)^{\underline{(N-1)}}\Theta((AC)^{N-1}A,1) =  \beta_N
\]
and the claim follows.
\end{proof}
Now that we know that Lemma~\ref{lemma:L17} is useful, we can now
continue proving it. The proof relies on Isserlis--Wick Theorem.
\begin{proof}[Proof of Lemma~\ref{lemma:L17}]
First we note that $\Theta(B,N) = \Theta(B^s,N)$ so we may assume that
$B$ is symmetric.

We recall and extend some notations from the introduction. We call multi-indices as \emph{words}.
For every word  $\rho = (\rho_j) \in \N^k$ of length $k$ and every
subset $J \subset \brc{1,2,\dots, k}$ of cardinality $m$ the word
$\Bar\rho_J$ is the word where the letters $\rho_j$ where $j \in J$ are
removed. Furthermore, for every permutation $\sigma$ on $J$ we will
denote $\rho_\sigma$ the word consisting the letters $\rho_j$ where $j
\in J$ in the order given by the permutation $\sigma$.

Using these notations with the definition of the $\Theta(B,N)$ we can
write
\[
\Theta(B,N) = \sum_\rho B_{\rho_{(1, 2)}}A^{\otimes
(N-1)}_{\Bar\rho_{\brc{1,2}}} \expectation \xi^{\otimes 2N}_\rho.
\]
By the Isserlis--Wick Theorem, the expectation can be written as a sum
\[
\expectation \xi^{\otimes 2N}_\rho = \sum_{k=2}^{2N} C_{\rho_{(1,k)}}
\expectation \xi^{\otimes 2(N-1)}_{\Bar\rho_{\brc{1,k}}}
\]
The term when $k = 2$ gives
\[
 \sum_\rho B_{\rho_{(1, 2)}}A^{\otimes
(N-1)}_{\Bar\rho_{\brc{1,2}}} \expectation \xi^{\otimes 2(N-1)}_{\Bar\rho_{\brc{1,2}}}
= (B : C) \Theta(N-1)
\]
which is the first term in the claim. The remaining terms can be written
as
\[
2\sum_{k=2}^N \sum_\rho B_{\rho_{(1, 2)}}A_{\rho_{(2k-1,2k)}} A^{\otimes
(N-2)}_{\Bar\rho_{\brc{1,2,2k-1,2k}}} C_{\rho_{(1,2k-1)}}
\expectation \xi^{\otimes 2(N-1)}_{\Bar\rho_{\brc{1,2k-1}}}
\]
by using the change of variables $\rho_{2k-1} \leftrightarrow \rho_{2k}$
and the symmetry of $A$. The change of variables $\rho_3 \leftrightarrow
\rho_{2k-1}$ and $\rho_4 \leftrightarrow \rho_{2k}$ show that the
previous sum reduces to
\[
\begin{split}
&2(N-1)\sum_\rho B_{\rho_{(1, 2)}}A_{\rho_{(3,4)}} A^{\otimes
(N-2)}_{\Bar\rho_{\brc{1,2,3,4}}} C_{\rho_{(1,3)}}
\expectation \xi^{\otimes 2(N-1)}_{\Bar\rho_{\brc{1,3}}}\\
&=2(N-1)\sum_\sigma \sum_{\rho_1, \rho_3} B_{\rho_1
\sigma_1}A_{\rho_3\sigma_2} A^{\otimes
(N-2)}_{\Bar\sigma_{\brc{1,2}}} C_{\rho_1\rho_3}
\expectation \xi^{\otimes 2(N-1)}_\sigma\\
&=2(N-1)\sum_\sigma BCA_{\sigma_{(1,2)}} A^{\otimes
(N-2)}_{\Bar\sigma_{\brc{1,2}}}
\expectation \xi^{\otimes 2(N-1)}_\sigma
\end{split}
\]
where we used the symmetricity of $B$. Since the last line coincides
with
\[2(N-1)\Theta(BCA, N-1)
\]
the claim follows.
\end{proof}
The proof of recursion equation (Lemma~\ref{lemma:L9}) is straightforward induction argument.
\begin{proof}[Proof of Lemma~\ref{lemma:L9}]
The result follows by induction. When $k = 1$, the result follows
immediately from the definition~\eqref{eq:D9} and the
Lemma~\ref{lemma:L8}.

Let us assume that the claim holds for $k < n-1$. Since
\[
\sum_{j=k}^n a_k(j,n) \Theta(n-j)
= a_k(k,n) \Theta(n-k) + \sum_{j=k+1}^n a_k(j,n) \Theta(n-j)
\]
we can expand the first term on the right-hand side with
Lemma~\ref{lemma:L8} and by change of summation variable $j' = j + k$ we
obtain
\[
\Theta(n)
= \sum_{j=k+1} a_k(k,n) a_1(j-k, n-k) + \sum_{j=k+1}^n a_k(j,n) \Theta(n-j)
\]
and the claim follows.
\end{proof}
Next we will solve the recursion equation for $a$ and therefore, the
$\Theta$. This is the content of Lemma~\ref{lemma:L13}.
We need some auxiliary functions to solve the recursion easily. First we denote 
\begin{equation}
  \label{eq:D10}
    I(n,m) = \bigcup_{j=1}^m J(n,j).
\end{equation}
With this we can define for $j \le n$ and $\mathbf k \in I(m, j)$
\begin{equation}
  \label{eq:D11}
  \Xi(\mathbf k, j, n) = a_1(j-s_m(\mathbf k), n-s_m(\mathbf k))
  \prod_{i=1}^m a_1(\mathbf k_i, n - s_{i-1}(\mathbf k)) 
\end{equation}
\begin{lem}
  \label{lemma:L12}
    We have for $1 \le k \le j \le n$ that
    \begin{equation}
      \label{L12:eq1}
      a_k(j,n) = \sum_{m=0}^{k-1} \sum_{\mathbf k} \Xi(\mathbf k, j, n)
      \Iverson{\mathbf k \in I(m, k-1)}.
    \end{equation}
    In particular,
    \[
    \Theta(n) = \sum_{m=0}^{n-1} \sum_{\mathbf k} \Xi(\mathbf k, n, n)
    \Iverson{\mathbf k \in I(m, n-1)}.
    \]
\end{lem}
Since we have an explicit formula for $\Xi(\mathbf k, n,n)$ and for
$a_1$ in terms of $R^{\mathbf k}$, we obtain a more explicit formula for
$\Theta$.
\begin{proof}[Proof of Lemma~\ref{lemma:L12}]
We show the identity~\eqref{L12:eq1} by induction with respect to $k$.
When $k = 1$, the identity follows from the fact that $I(0,0) =
\{\emptyset\}$ consists of a single element, namely the empty word
$\emptyset$. Since $s_0(\emptyset) = 0$ and the empty product is $1$, we notice
that the right-hand side of~\eqref{L12:eq1} reduces to
\[
\Xi(\emptyset, j, n) = a_1(j-s_0(\emptyset), n-s_0(\emptyset)
) \prod_{\emptyset} \dots = a_1(j,n)
\]
and the claim holds for $k=1$.

Let us now suppose that the claim holds for $k < n$. By Lemma~\ref{lemma:L9} we have
\begin{equation}
  \label{L12:eq2}
\begin{split}
a_{k+1}(j,n) &= a_k(k,n)a_1(j-k,n-k) + a_k(j,n)\\
&= \sum_{m=0}^{k-1} \sum_{\mathbf k \in I(m,k-1)} 
\Xi(\mathbf k, j, n) + \Xi(\mathbf k, k,n)a_1(j-k,n-k)
\end{split}
\end{equation}
where the last identity follows by the induction assumption. When
$\mathbf k \in I(m,k-1)$ we define a new word $\overline{\mathbf k}$ by
adding a single letter in the end
\[
\overline {\mathbf k} = \mathbf k \; \& \; \brc{k-s_m(\mathbf k)}.
\]
We notice that the mapping $\mathbf k \mapsto \overline{\mathbf k}$
defines a bijection from $I(m,k-1)$ onto $J(m+1,k)$. Furthermore, we
notice that for $\mathbf k \in I(m,k-1)$ we have
\[
\begin{split}
\Xi(\mathbf k, k,n)a_1(j-k,n-k) & = a_1(j-k,n-k) \prod_{i =
1}^{m+1} a_1(\overline{\mathbf k}_i, n-s_{i-1}(\overline{\mathbf k}))\\
& = \Xi(\overline{\mathbf k}, j, n).
\end{split}
\]
Therefore, since $\mathbf k \mapsto \overline{\mathbf k}$ is a
bijection, we obtain
\begin{equation}
  \label{L12:eq4}
\sum_{m=0}^{k-1} \sum_{\mathbf k \in I(m,k-1)} 
\Xi(\mathbf k, k,n)a_1(j-k,n-k)
= \sum_{m=1}^{k} \sum_{\mathbf k \in J(m,k)} 
\Xi(\mathbf k, j, n)
\end{equation}
Moreover, since $I(0,k) = \brc{\emptyset}$ and $I(k,k) = J(k,k)$ and
\[
\Iverson{\mathbf k \in I(m,k-1)} + \Iverson{\mathbf k \in J(m,k)}
= \Iverson{\mathbf k \in I(m,k)}
\]
the induction hypothesis follows by combining identities~\eqref{L12:eq2}
and~\eqref{L12:eq4}. This proves the claim.
\end{proof}
We can apply the previous lemma (Lemma~\ref{lemma:L12}) to solve the
recursion equation (Equation~\eqref{eq:D9}) for $\Theta$.
\begin{proof}[Proof of Lemma~\ref{lemma:L13}]
By Lemma~\ref{lemma:L12} we have
\[
\Theta(n) = \sum_{m=0}^{n-1} \sum_{\mathbf k} \Xi(\mathbf k, n, n)
\Iverson{\mathbf k \in I(m, n-1)}.
\]
We note that for $\mathbf k \in I(m,n-1)$ it holds that
\[
\Xi(\mathbf k,n,n) = \prod_{i=1}^{m+1} a_1(\overline{\mathbf k_j},
n-s_{i-1}(\overline{\mathbf k}))
\]
where $\overline{\mathbf k} = \mathbf k \; \& \; \brc{n-s_m(\mathbf k)}$ as in
the proof of Lemma~\ref{lemma:L12}. Since the mapping $\mathbf k \mapsto
\overline{\mathbf k}$ is a bijection from $I(m,n-1)$ onto $J(m+1,n)$ the
claim follows by using the facts that
\[
a_1(j,n) = 2^{j-1}(n-1)^{\underline{j-1}} R_j
\]
and the definition of the functions $s_j$.
\end{proof}
The representation of Lemma~\ref{lemma:L13} can be used to obtain the representation
of given by Lemma~\ref{lemma:L15}.
\begin{proof}[Proof of Lemma~\ref{lemma:L15}]
The Lemma~\ref{lemma:L13} immediately implies that
\[
\Theta(n) = \sum_{m=1}^{n} \sum_{l=0}^m  2^{n-m}{R_1^l}\sum_{\mathbf k \in J_l(m,n)} 
\prod_{j=1}^{m} q(j,\mathbf k_j, \mathbf k)
\]
where 
\[
q(j,k, \mathbf k) = (R_k \Iverson{k \ne 1} + \Iverson{k = 1})
(n-s_{j-1}(\mathbf k)-1)^{\underline{k - 1}}.
\]
Note that every $\mathbf k \in J_l(m,n)$ can be uniquely represented
by giving the locations and values of the indices different from $1$.
In particular, there is a bijection $\pi = (\pi_j)$ from
\[
J_0(m-l,n-l) \times L(m-l,m) \to J_l(m,n)
\]
given by
\[
\pi_k(\Bar{\mathbf k}, \lambda) = 1 + \sum_j (\Bar{\mathbf k}_j-1)
\Iverson{\lambda_j = k}.
\]
Therefore,
\[
\begin{split}
&\sum_{\mathbf k \in J_l(m,n)}\prod_{j=1}^m q(j,\mathbf k_j, \mathbf k)\\
&= 
\sum_{\Bar{\mathbf k} \in J_0(m-l,n-l)}\sum_{\lambda \in L(m-l,m)}\prod_{1 \le j \le
m-l} q(\lambda_j,\Bar{\mathbf k_j}, \pi(\Bar{\mathbf k}, \lambda))\\
&= 
\sum_{\mathbf k \in J_0(m-l,n-l)} R^{\mathbf k}
\Lambda(n,\mathbf k)
\end{split}
\]
and the claim follows.
\end{proof}
In order to prove Lemma~\ref{lemma:L16} we need some more auxiliary
results. 
We first introduce the word length function
\begin{equation}
  \label{eq:D13}
  \psi(\mathbf k) = \text{"length of $\mathbf k$"}
\end{equation}
Next, we denote
\begin{equation}
  \label{L15:eq1}
w(n, \lambda, \mathbf k) := \prod_{j=1}^{\psi(\mathbf k)}
(n-(\mathbf\lambda_j-j) - s_{j-1}(\mathbf k) - 1)^{\underline{\mathbf k_j -1}}
\end{equation}
for every $\lambda \in L(\psi(\mathbf k), m)$.

The auxiliary function $w$ can be written in a closed form with the help
of induction. We will provide the large step reduction lemma, small step
reduction and the base step in auxiliary lemmata.
\begin{lem}
  \label{lemma:L20}
  Let $M = m-l \ge 1$.
  For every $\mathbf k \in J_0(M,n-l)$ we have
  \[
  \sum_{\lambda \in L(M, m)} w(n, \lambda, \mathbf k)
  = \binom{n}{n-l}(n-l-1)!
  \prod_{2 \le j \le M} \frac {1}{s_M(\mathbf k)-s_{j-1}(\mathbf k)}
  \]
\end{lem}
In order to show this we need a reduction lemma that reduces the length of word $\mathbf
k$.
\begin{lem}
  \label{lemma:L18}
    Let $M = m - l > 1$.
    For every $\mathbf k \in J_0(M,n-l)$ we have
    \[
    \sum_{\lambda \in L(M,m) }w(n,\lambda, \mathbf k) 
    = \mathbf k_{M}^{-1}
    \sum_{\lambda \in L(M-1,m-1)} w(n,\lambda, \Bar{\mathbf k})
    \]
    where $\Bar{\mathbf k} \in J_0(M-1,n-l)$ and is defined as
    \[
    \Bar{\mathbf k_j} := \Iverson{j < M-1} \mathbf k_j + \Iverson{j =
    M-1}(\mathbf k_M + \mathbf k_{M-1})
    \]
\end{lem}
Furthermore, we need the base step lemma for one letter words.
\begin{lem}
  \label{lemma:L19}
    For every $\mathbf k \in J_0(1,n-l)$ we have
    \[
    \sum_{\lambda \in L(1,l+1)}w(n,\lambda, \mathbf k) 
    = \binom n{\mathbf k_1} (\mathbf k_1 - 1)!
    \]
\end{lem}
Now we can prove the representation lemma for the auxiliary function
$w$.
\begin{proof}[Proof of Lemma~\ref{lemma:L20}]
  When $M = 1$, we have $m = l+1$ and $\mathbf k_1 = n-l$. Therefore,
  the claim follows from Lemma~\ref{lemma:L19}.

  Suppose the claim holds for $M = M_0 \ge 1$ and consider the case $M =
  M_0 + 1$. In this case, $m = M + l = M_0 + l + 1$. By
  Lemma~\ref{lemma:L18} we have
  \[
  \begin{split}
  \sum_{\lambda \in L(M,m) }w(n,\lambda, \mathbf k) 
  &= \mathbf k_{M}^{-1}
  \sum_{\lambda \in L(M_0,M_0+l)} w(n,\lambda, \Bar{\mathbf k})\\
  &= \mathbf k_{M}^{-1}\binom{n}{n-l}(n-l-1)!
  \prod_{2 \le j \le M_0} \frac {1}{s_{M_0}(\Bar{\mathbf k})-s_{j-1}(\Bar{\mathbf k})}
  \end{split}
  \]
  Since $s_{M_0}(\Bar{\mathbf k}) = s_M(\mathbf k)$ and
  $s_j(\Bar{\mathbf k}) = s_j(\mathbf k)$ for every $j < M_0$ and
  moreover, $\mathbf k_M = s_M(\mathbf k) - s_{M-1}(\mathbf k)$ the
  induction claim follows and the claim is proved.
\end{proof}
Next we prove the base step.
\begin{proof}[Proof of Lemma~\ref{lemma:L19}]
In this case, we have
\[
w(n,\lambda, \mathbf k) = (n-j)^{\underline{k-1}}
\]
where $j = \lambda_1 \in \brc{1,\dots, l+1}$ and $k = \mathbf k_1$.
Since $\mathbf k \in J_0(1,n-l)$ we have $k = n-l$. Therefore the sum in
this case reduces to
\[
\sum_{j = 1}^{l+1} (n-j)^{\underline {k-1}} = \sum_{j = n-l-1}^{n-1}
j^{\underline {k-1}} = \sum_{j = k}^{n-1} j^{\underline{k-1}}
= (k-1)! \sum_{j=k}^{n-1} \binom j{k-1}
\]
and since the last sum equals to $\binom nk$ the claim follows.
\end{proof}
The reduction lemma (Lemma~\ref{lemma:L18}) will be shown next.
\begin{proof}[Proof of Lemma~\ref{lemma:L18}]
Since $n$ is fixed throughout the proof, we will drop it from the
argument lists of functions.

We split the $\lambda \in L(M,m)$ into two parts, i.e. we write $\lambda
= \lambda'\, \& \, \lambda_M$ where $\lambda' \in L(M-1,m-1)$. This
implies that
\[
\sum_{\lambda} w(\lambda, \mathbf k) = \sum_{\lambda'} w(\lambda',
\mathbf k') \sum_{k = \lambda'_{M-1}}^{m-1} w^*(\mathbf k, k+1) 
\] 
where
\[
w^*(\mathbf k, k) 
:= (n-(k -M) - s_{M-1}(\mathbf k) -1)^{\underline{\mathbf k_M-1}}.
\]
We note that the sum is of form
\[
\sum_{k = \alpha_1}^{\alpha_2} (\alpha_3 - k)^{\underline{\alpha_4}}.
\]
which can be computed easily if $\alpha_4 = \alpha_3-\alpha_2$, since
then
\[
\sum_{k = \alpha_1}^{\alpha_2} (\alpha_3 - k)^{\underline{\alpha_4}}
= \alpha_4! \binom{\alpha_3 - \alpha_1 + 1}{\alpha_4 + 1}
= (\alpha_4+1)^{-1} ({\alpha_3 - \alpha_1 + 1})^{\underline{\alpha_4 +
1}}
\]
In this case the condition $\alpha_3 - \alpha_2 = \alpha_4$ is
equivalent with
\[
(n + M - s_{M-1}(\mathbf k) -1) - m = \mathbf k_M - 1
\]
This follows from the fact that $\mathbf k \in J_0(M,n-l)$ and $M > 1$,
since this implies that
\[
s_{M-1}(\mathbf k) = n-l - \mathbf k_M.
\]
Therefore, the condition is equivalent with $M = m - l$ which holds by
the assumption. We can combine the falling product to the last falling
product in $w(\lambda', \mathbf k')$ which can be written as
\[
w(\lambda'', \mathbf k'') (n - (\lambda_{M-1}' - (M-1)) -
s_{M-2}(\mathbf{k}) - 1)^{\underline{\mathbf k_{M-1} -1}}.
\]
The last factor in this falling product is
\[
J = (n - (\lambda_{M-1}' - (M-1)) - s_{M-1}(\mathbf{k})  + 1)
\]
since $s_{M-1} = s_{M-2} + \mathbf k_{M-1}$. On the other hand, the
first factor in the falling product of $w^*$ is
\[
\alpha_3 - \alpha_1 + 1 = (n - (\lambda_{M-1}' - M) - s_{M-1}(\mathbf k)
- 1) = J-1.
\]
Thus, the falling factors can be combined into a single falling factor
of length $\mathbf k_M + \mathbf k_{M-1} - 1$ and the claim follows.
\end{proof}
The last missing piece of the Section~\ref{SLLNGQF} is the proof of
Lemma~\ref{lemma:L16}.
\begin{proof}[Proof of Lemma~\ref{lemma:L16}]
We begin the proof with few observations and notations. First, let us
start with fixed $\mathbf\lambda \in L(m-l,m)$ and $\mathbf k \in
J_0(m-l,n-l)$. We will denote the word $\pi(\mathbf k, \mathbf\lambda)$ just by
$\pi$ for awhile. Let us denote the left-inverse of $j \mapsto
\mathbf\lambda_j$ by $\delta$ i.e. we define
\[
\delta_j = \max\set{k}{\mathbf\lambda_k \le j}.
\]
We also observe that 
\[
\delta_j = \sum_{k=1}^j \Iverson{\pi_k \ne 1}
\]
which implies that
\[
\sum_{k=1}^j \Iverson{\pi_k = 1} = j - \delta_j.
\]
Thus, for every $j$
\[
s_j(\pi) = \sum_{k=1}^j \Iverson{\pi_k = 1} + \sum_{k=1}^j \pi_k
\Iverson{\pi_k \ne 1}
= (j - \delta_j) + s_{\delta_j} (\mathbf k) 
\]
Since $\delta_{\mathbf\lambda_j -1} = j-1$, we obtain
\[
s_{\mathbf\lambda_j-1}(\pi) = (\mathbf\lambda_j - 1 - (j-1)) + s_{j-1}
(\mathbf k).
\]
Therefore, we have shown that
\[
\prod_{j=1}^{m-l} (n-s_{\mathbf\lambda_j -1}(\pi)
-1)^{\underline{\mathbf k_j -1}} 
= w(n,\lambda, \mathbf k)
\]
since $\psi(\mathbf k) = m-l$.

We can now sum over all $\mathbf\lambda$'s and we obtain
\[
\Lambda(n,\mathbf k) = \sum_{\lambda \in L(M,m)} w(n, \lambda,\mathbf k)
\]
and the claim follows from Lemma~\ref{lemma:L20}.
\end{proof}

\section{Proofs of auxiliary results in Section~\ref{SLLN}}
\label{PFB}

We gather here the proofs of auxiliary lemmata we used in the Section~\ref{SLLNGQF}.
 We start with Lemma~\ref{lemma:B:L2} and
Lemma~\ref{lemma:B:L5} that deal with the matrices $K_n$ and $\widetilde
K_n$.
\begin{proof}[Proof of Lemma~\ref{lemma:B:L2}]
  This follows by analysing the corresponding properties of the
  convolution operator $C(g)$ corresponding to a symbol $g$. For
  convolution operators, we can show that
  \[
  C(g)C(h) = C(gh)
  \]
  whenever $\widehat g \asymp c_\alpha$, $\widetilde h \asymp c_\beta$ and
  $\alpha + \beta < \puokki$. Since the Fourier coefficients of
  $g^{-1}_\alpha$ (which follows Lemma~\ref{lemma:C:L6}) behave asymptotically as the Fourier coefficients of
  $g_{-\alpha}$, we deduce
  \[
  C(g_\alpha)C(1/g_\alpha) = I
  \]
  Expressing the convolution operator as an element of $\R^{\Z \times
  \Z}$ and dividing it into 9 blocks we have
  \[
  \begin{pmatrix}
  \cdot & \cdot & \cdot\\
  A_{21} & A_{22} & A_{23}\\
  \cdot & \cdot & \cdot
  \end{pmatrix}
  \begin{pmatrix}
  \cdot & B^*_{21}& \cdot\\
  \cdot & B_{22} & \cdot\\
  \cdot & B^*_{23} & \cdot
  \end{pmatrix}
  =
  \begin{pmatrix}
  \cdot & \cdot & \cdot\\
  \cdot & I_n & \cdot\\
  \cdot & \cdot & \cdot
  \end{pmatrix}
  \]
Since $A_{22} = T_n(g_\alpha)$ and $B_{22} = T_n(g^{-1}_\alpha)$ we know
$A_{22}$ is invertible and therefore
\[
B_{22} = A_{22}^{-1} - B_{22}(A_{21}B^*_{21} + A_{23}B^*_{23})
\]
If we denote $K_n(\alpha) = A_{21}B^*_{21} + A_{23}B^*_{23}$, the claim
follows by symmetrizing the identity. 
\end{proof}
\begin{proof}[Proof of Lemma~\ref{lemma:B:L5}]
We will drop subscript $n$ from the following unless it is essential. We
will denote the Fourier coefficients of $g^{-1}_\alpha$ by $d_j$ and we
will denote $c_j = c_\alpha(j)$.
From the proof of Lemma~\ref{lemma:B:L2} we know that
\[
K_{ij} = \sum_{l = 1}^{\infty} (c_{i+l} d_{j+l} +
c_{n-i+l}d_{n-j+l}) = A_{ij} + A_{(n-i)(n-j)}.
\]
Therefore, we may estimate
\[
2\widetilde K_{ij} \le |B|_{ij} + |B|_{(n-i)(n-j)}
\]
when $B = A + A^\top$. Since $|c_{i+l}| \asymp (i+l)^{2\alpha - 1} \asymp (i
\vee l)^{2\alpha - 1}$ and analogously $|d_{j+l}| \asymp (j \vee
l)^{-2\alpha - 1}$, we can estimate
\[
|B|_{ij} \lesssim n^{-1} \sum_{\pm} \int_{1/n}^\infty (x \vee t)^{-1\pm2\alpha}(t
\vee y)^{-1\mp 2\alpha} \di t.
\]
We notice that the right-hand side stays invariant in the transformations
$\alpha \leftrightarrow -\alpha$ and $x \leftrightarrow y$, so we may
assume that $\alpha = |\alpha| > 0$ and $x \le y$. We will denote $v_t =
t/y$ whenever $t \le y$ and in particular, when $t = x$, we will denote
$w = v_x$. Moreover, it holds that $v_t^{2\alpha} + v_t^{-2\alpha} \lesssim
v_t^{-2\alpha}$ for all $t \le y$.

When $t \le x$ we have $t \vee x = x$ and $t \vee y = y$ and therefore
\[
\sum_{\pm} \int_{1/n}^x \dots \di t \lesssim y^{-1}( w^{2\alpha} +
w^{-2\alpha} ) \lesssim y^{-1} w^{-2\alpha}
\]
When $x < t \le y$, we have $t \vee x = t$ and $t \vee y = y$. Hence
\[
\begin{split}
\sum_{\pm} \int_{x}^y \dots \di t & \lesssim y^{-1} \int_x^y t^{-1} ( v_t^{2\alpha} +
v_t^{-2\alpha} ) \di t \lesssim y^{-1} \int_x^y t^{-1} v_t^{-2\alpha}
\di t \\
& \lesssim y^{-1}w^{-2\alpha}.
\end{split}
\]
The remaining part has a trivial upper bound $2 y^{-1} \lesssim y^{-1}
w^{-2\alpha}$. Combining these three cases the claim follows.
\end{proof}
Lemma~\ref{lemma:B:L5} gives enough control for proving
Proposition~\ref{lemma:B:L4}.
\begin{proof}[Proof of Proposition~\ref{lemma:B:L4}]
Since $\alpha$ and $\beta$ are fixed throughout the proof, we will
drop them from the subscripts.

When $\alpha = 0$, the $\widetilde K_n = 0$ and the claim is trivial.
Therefore, we may suppose $\alpha \ne 0$ and since $\widetilde
K_n(\alpha)$ only depends on the absolute value of $\alpha$, we may
assume $\alpha > 0$ as well.

When $\beta = 0$, the Toeplitz matrix $T_n = I$. Therefore, 
\[
|T_n| : \widetilde K_n \asymp \int_{1/n}^{1-1/n} k(x,x) \di x \asymp
\int_{1/n}^1 x^{-1} \di x \asymp \log n
\]
which implies the claim in this case. So we may assume that $\beta \ne
0$ in the sequel.

We have an asymptotic representation
\begin{equation}
  \label{B:L5:e1}
  |T_n|_{ij} \asymp n^{-1} |x-y|^{2\beta-1} \Iverson{|x - y|>n^{-1}} +
  \Iverson{i=j}.
\end{equation}
Since we already computed the claim for identity matrix, we may
concentrate to contribution coming from outside the main diagonal.

By Lemma~\ref{lemma:B:L5}, the representation~\eqref{B:L5:e1} and the
symmetry $(x,y) \leftrightarrow (1-x, 1-y)$ we see that
\[
|T_n| : \widetilde K_n \lesssim n^{2\beta} \int_{I_n} \frac{(x\vee
y)^{-1+2\alpha}}{(x \wedge y)^{2\alpha}} |x-y|^{-1+2\beta}
\di x\di y = n^{2\beta} \int_{I_n} f.
\]
where $I_n = \brc{|x-y| > n^{-1}, x \wedge y \ge n^{-1}}$. Let us keep
$y$ fixed first. If we suppose $y/2 < x < 2y$, we have an estimate $x\wedge y
\asymp x \vee y \asymp y$.
Therefore, we have
\[
\begin{split}
\int_{y/2}^{2y} f(x,y) \Iverson{(x,y) \in I_n} \di x &\asymp
y^{-1}\int_{1/n}^{y \wedge (1-y)}
x^{2\beta - 1} \di x
\end{split}
\]
Considering the cases $y < \puokki$ and $y \ge \puokki$ separately, we
obtain
\[
\int_0^1 \di y \int_{y/2}^{2y} f(x,y) \Iverson{(x,y) \in I_n} \di x \asymp
n^{-2\beta} \log n\Iverson{\beta < 0} + \Iverson{\beta > 0}.
\]

When $x \le y/2$, we have an estimate $|x-y| \asymp y$. In this case the
integral reduces to
\[
\int_0^{y/2} f(x,y) \Iverson{(x,y) \in I_n} \di x \di y 
\asymp  y^{-1+2\beta}
\]
since $\alpha < \puokki$. When $x \ge 2y$, we have an estimate $|x-y|
\asymp x$. In this case the 
integral can therefore be estimated as
\[
\int_{2y}^1 f(x,y) \Iverson{(x,y) \in I_n} \di x \di y 
\asymp \Iverson{y < \puokki} (y^{-1+2\beta} \vee y^{-2\alpha}).
\]
Integrating these two last cases with respect to $y$ and summing all the
cases together shows that
\[
n^{2\beta} \int_{I_n} f(x,y) \di x \di y \asymp \log n \Iverson{\beta < 0}
+ n^{2\beta} \Iverson{\beta > 0}\asymp n^{2\beta} \vee \log n
\]
and the claim follows.
\end{proof}
\begin{proof}[Proof of Lemma~\ref{lemma:B:L6}]
  By Lemma~\ref{lemma:C:L3} and the reasoning explained in
  Section~\ref{factorisation} we know that $|T_n(g_{-\alpha})^{-1}|$
  behaves elementwise as 
  $|T_n(\theta_{2\alpha})^{-1}|$. The diagonal estimate follows from
  \cite[Théorème 1]{RambourSeghier2}. Outside the diagonal, we divide
  the proof in two parts $\alpha > 0$ and $\alpha < 0$. Since $x$ and
  $y$ will be fixed throughout the proof, we will usually drop them from
  parameters of functions for notational simplicity.

  When $\alpha > 0$, we have
  \[
  T^{-1}_{ij} \asymp n^{-1+2\alpha} \mathcal S(f)(x,y)
  \]
  where the function $f$ in the triangle $x \vee \wt x \le y < 1$ is
  given by
  \[
  f(x,y) = x^\alpha y^\alpha \int_y^1 \rho(t) t^{-2\alpha} \di t \asymp 
   x^\alpha \int_y^1 \rho(t) \di t.
  \]
  Here and later we will denote
  \[
  \rho(t)  = (t-x)^{\alpha -1} (t-y)^{\alpha -1}.
  \]
  The function $\rho$ satisfies
  \[
  \rho(t) \asymp 
  \Iverson{t < z}(t-y)^{\alpha -1} w^{\alpha - 1}
  + \Iverson{t \ge z}(t-y)^{2\alpha -2}
  \]
  where $z = 2y-x$ and $w = y-x$.
  When $y \in I_b(x)$ we have $\Iverson{t \ge z} = 0$ and
  therefore,
  \[
  f \asymp w^{\alpha - 1}x^\alpha \widetilde y^\alpha =
  E_2^{(\alpha)}.
  \]
  When $y \in I_d(x)$, we have $z \le 1$ and in this case
  \[
  \int_y^1 \rho(t)\di t \asymp w^{\alpha - 1}\int_0^w t^{\alpha
  -1} \di t + 
  \int_w^{\widetilde y} t^{2\alpha - 2}\di t \asymp w^{2\alpha -
  1},
  \]
  giving the claim for $\alpha > 0$.

  When $\alpha < 0$, we similarly have
  \[
  T^{-1}_{ij} \asymp n^{-1+2\alpha} \mathcal S(f + f_2)(x,y)
  \]
  where the functions $f$ and $f_2$ in the triangle $x \vee \wt x \le y
  < 1$ are
  given by
  \[
  \begin{split}
  f(x,y) &= -x^\alpha y^\alpha \int_y^1 \rho(t)
  \big(\rho_2(t)-\rho_2(y)\big) t^{-2\alpha} \di t\\
  \text{and}\quad f_2(x,y) &= \alpha^{-1} x^\alpha \widetilde y^\alpha
  y^{-\alpha} w^{\alpha -1} \asymp -x^\alpha \widetilde y^\alpha
  w^{\alpha -1}.
  \end{split}
  \]
  The auxiliary function $\rho_2$ is given by
  \[
  \rho_2(s) = \Big(\frac w{s-x}\Big)^{\alpha -1} \Big(\frac
  sy\Big)^{2\alpha}.
  \]
  We note that
  \begin{equation}
    \label{B:E11}
  \begin{split}
  \rho_2'(s) &= \big((\alpha+1)s-2\alpha
  x\big)s^{-1}(s-x)^{-1}\rho_2(s)
  \asymp w^{\alpha -1}(s-x)^{-\alpha}\\
  & \asymp \Iverson{s < z} w^{-1} + \Iverson{s \ge z}w^{\alpha -1}
  (s-y)^{-\alpha}.
  \end{split}
  \end{equation}
  When $y \in I_b(x)$ we have $\Iverson{s \ge z} = \Iverson{t \ge z} = 0$ and therefore,
  \[
  \rho(t)\big(\rho_2(t)-\rho_2(y)\big)t^{-2\alpha} \asymp \rho(t)\int_y^t
  w^{-1} \di s \asymp (t-y)^{\alpha}w^{\alpha-2}
  \]
  and hence for every $y \in I_b(x)$ we have
  \[
  0 \le -f(x,y) \lesssim x^\alpha \widetilde y^{\alpha +1} w^{\alpha
  - 2} \lesssim -f_2(x,y)
  \]
  where in the last estimate we used the fact that $\widetilde y
  w^{-1} \le 1$ if and only if $y \in I_b(x)$. Thus, $f + f_2
  \asymp -E_2^{(\alpha)}$ whenever $y \in I_b(x)$.

  When $y \in I_d(x)$ we have more cases. First we note that
  $\Iverson{s < z} \ge \Iverson{t < z}$ and therefore,
  \[
  \Iverson{t < z}\rho(t)\big(\rho_2(t)-\rho_2(y)\big)t^{-2\alpha} \asymp 
  \Iverson{t < z} (t-y)^{\alpha}w^{\alpha-2}
  \]
  This leads to 
  \[
  \int_y^z \rho(t)\big(\rho_2(t)-\rho_2(y)\big)t^{-2\alpha} \di t\asymp 
  w^{-1+2\alpha}
  \]
  When $t \ge z$, we use a cruder estimate of $\rho_2'$ by estimating the
  indicators functions on the right hand side above by constants which gives an estimate
  \[
  0 \le \rho(t)\big(\rho_2(t)-\rho_2(y)\big)t^{-2\alpha} \lesssim
  w^{-1}(t-y)^{2\alpha -1} + w^{\alpha -1}(t-y)^{\alpha-1}
  \]
  Since $(t-y)^{\alpha} \le w^\alpha$ this implies that 
  \[
  0 \le \int_z^1 \rho(t) \big(\rho_2(t)-\rho_2(y)\big)t^{-2\alpha}\di t\lesssim
  w^{\alpha-1}\int_{w}^{\widetilde y} t^{\alpha -1} \di t \lesssim w^{-1+2\alpha}
  \]
  Therefore, $f(x,y) \asymp -x^\alpha w^{-1+2\alpha}$. Since
  \[
  0 \le -f_2 \lesssim (\widetilde y w^{-1})^{\alpha} f\le f
  \]
  whenever $y \in I_d(x)$ the claim follows.
\end{proof}

\begin{lem}
  \label{lemma:42:A1}
    Let $\alpha, \beta \in (-\puo, \puo)$ and denote $\nu = \alpha_- + \beta_+$.
    We have the following asymptotic estimates:
    \begin{itemize}
    \item[$-$]
    For every $\gamma = 2\beta \ne 1$ it holds
    \[
    \tag{A}\label{lemma:42:A1:A}
    n^{\gamma} \int_{I_n^2} \abs{x_1-x_2}_n^{-2+\gamma} \dd x_1 \dd x_2
    \asymp n^{\gamma \vee 1}
    \]
    \item[$-$]
    For every $\alpha \ne 0$ and $\beta \ne 0$ it holds
    \[
    \tag{B}\label{lemma:42:A1:B}
    \int_{I_n^2} \abs{x_1-x_2}_n^{-1+2\beta} x_1^{-1-\alpha} \dd x_1 \dd x_2
    \asymp n^{2\beta_- + \alpha_+}
    \]
    \item[$-$]
    For every $x_1, x_3 \in I_n$ such that $x_3 - x_1 > 3n^{-1}$ and
    for $\gamma = \nu-\puo$ it holds
    \[
    \tag{C}\label{lemma:42:A1:C}
    \begin{split}
    &\int_{n^{-1}}^{\frac{x_1+x_3}2} \abs{x_1-x_2}_n^{-1+2\beta}
    \abs{x_2-x_3}^{-1-2\alpha} \dd x_2\\
    &\asymp \Iverson{\gamma < 0} n^{2\beta_-}\abs{x_1-x_3}_n^{-1 +
    2(\beta_+ - \alpha)} + \Iverson{\gamma > 0}x_1^{2\gamma}
    \end{split}
    \]
    \end{itemize}
\end{lem}
\begin{proof}[Proof of Proposition~\ref{lemma:B:L1}]
In order to prove the claim, we first use the representation of the
inverse matrix $T_n(g_\alpha)^{-1}$ from Lemma~\ref{lemma:B:L6} and use
the triangle inequality for norms to conclude that
\[
\begin{split}
&\norm{T_n(g_\beta)^{\nicefrac1{2}}T_n(g_\alpha)^{-1} T_n(g_\beta)^{\nicefrac1{2}}}_F
\\
& \qquad \asymp 
\norm{T_n(g_\beta)^{\nicefrac1{2}} \big(I + \widetilde E_1^{(-\alpha,n)} + \widetilde
E_2^{(-\alpha)} \big) T_n(g_\beta)^{\nicefrac1{2}}}_F\\
& \qquad \le \norm{T_n(g_\beta)}_F 
+ 
\norm{T_n(g_\beta)^{\nicefrac1{2}} \widetilde E_1^{(-\alpha,n)} 
T_n(g_\beta)^{\nicefrac1{2}}}_F\\
& \qquad + 
\norm{T_n(g_\beta)^{\nicefrac1{2}} \widetilde E_2^{(-\alpha)} 
T_n(g_\beta)^{\nicefrac1{2}}}_F
\end{split}
\]
The first term on the right-hand side is handled directly with the
estimate~\eqref{lemma:42:A1:A} from Lemma~\ref{lemma:42:A1} since
\[
\norm{T_n(g_\beta)}^2_F \asymp n^{4\beta}
\int_{I_n^2} \abs{x-y}^{-2+4\beta} \dd x\dd y \asymp n^{4\beta \vee 1}
\]
when $\beta \ne \smallFrac14$.

The second term can be estimated with the help of
Lemma~\ref{lemma:42:A3} which yields that
\[
\norm{T_n(g_\beta)^{\nicefrac1{2}} \widetilde E_1^{(-\alpha,n)} T_n(g_\beta)^{\nicefrac1{2}}}_F
\asymp n^{2(\alpha_- + \beta_+) \vee \nicefrac1{2}}
\]
This already implies that the sum of the first two terms gives the
claimed asymptotics. We still need to show that the third term has at
most the claimed growth properties but in order to do that we need to
split the symbol into two parts and we do this by estimating
\[
\begin{split}
\norm{T_n(g_\beta)^{\nicefrac1{2}} \widetilde E_2^{(-\alpha)} 
T_n(g_\beta)^{\nicefrac1{2}}}_F
& \le 
\norm{T_n(g_\beta)^{\nicefrac1{2}} \widetilde E_{3}^{(-\alpha)}
T_n(g_\beta)^{\nicefrac1{2}}}_F\\
&+
\norm{T_n(g_\beta)^{\nicefrac1{2}} \widetilde E_{4}^{(-\alpha)}
T_n(g_\beta)^{\nicefrac1{2}}}_F
\end{split}
\]
Here $\widetilde E_{3}^{(-\alpha)}$ and $\widetilde
E_{4}^{(-\alpha)}$ are $n \times n$ -matrices corresponding to 
kernels $\mathcal S(E_2^{(-\alpha)} \Iverson{x \le \twothirds})$
and $\mathcal S(E_2^{(-\alpha)} \Iverson{x > \twothirds})$,
respectively. 
Therefore, using
Lemmata~\ref{lemma:42:A5} and~\ref{lemma:42:A6} the claim follows.
\end{proof}
\begin{proof}[Proof of Proposition~\ref{lemma:B:L3}]
Throughout the proof we assume $\alpha \ne 0$.  The proof of this
estimate follows from the following estimates
\begin{equation}
  \label{eq:p:B:L3:1}
    I \widetilde K_n(\alpha) \lesssim n^{2\alpha_-} \widetilde
    K_n(\alpha)
\end{equation}
\begin{equation}
  \label{eq:p:B:L3:2}
  \begin{split}
    n^{-1-2\alpha} &\int_{n^{-1}}^{1-n^{-1}} \abs{x-t}_n^{-1-2\alpha}
    k_\alpha(t,y) \Iverson{ t \in \mathcal S(I_d(x))} \dd t\\
    & \asymp n^{-1-2\alpha_-} k_\alpha(x,y)
  \end{split}
\end{equation}
and
\begin{equation}
  \label{eq:p:B:L3:3}
  \begin{split}
    n^{-1-2\alpha} &\int_{n^{-1}}^{1-n^{-1}} \mathcal
    S(E_2^{\alpha})(x,t) 
    k_\alpha(t,y) \Iverson{ t \notin \mathcal S(I_d(x))} \dd t\\
    & \lesssim n^{-1-2\alpha_-} k_\alpha(x,y).
  \end{split}
\end{equation}
Summing these estimates together with the representation of
the matrix $\abs{T_n(g_{-\alpha})^{-1}}$ given by Lemma~\ref{lemma:B:L6} implies
the claim. Moreover, we notice from the asymptotic
estimate~\eqref{eq:p:B:L3:2} that the estimate is actually sharp and we
the same estimate for the lower bound.

The first estimate~\eqref{eq:p:B:L3:1} is trivial and the latter two
estimates~\eqref{eq:p:B:L3:2} and~\eqref{eq:p:B:L3:3} are
given by the auxiliary lemmata~\ref{lemma:44:B} and~\ref{lemma:44:C}
respectively.
\end{proof}
\begin{lem}
  \label{lemma:42:A3}
  For every $\alpha, \beta \in (-{\nicefrac1{2}}, {\nicefrac1{2}})$ we have that
  \[
  \norm{T_n(g_\beta)^{\nicefrac1{2}} \widetilde E_1^{(-\alpha,n)} T_n(g_\beta)^{\nicefrac1{2}}}_F
  \asymp n^{2(\alpha_- + \beta_+) \vee \nicefrac1{2}}
  \]
\end{lem}
\begin{proof}
The squared norm has an asymptotic estimate
\[
\begin{split}
  &\norm{T_n(g_\beta)^{\nicefrac1{2}} \widetilde
  E_1^{(-\alpha,n)} T_n(g_\beta)^{\nicefrac1{2}}}_F^2
  \asymp
  n^{4(\beta-\alpha)}
    \int_{I_n^4} 
    \abs{x_1-x_2}_n^{-1+2\beta}  \times\\
    &\qquad\qquad \times
    \abs{x_2-x_3}_n^{-1-2\alpha}
    \abs{x_3-x_4}_n^{-1+2\beta}
    \abs{x_4-x_1}_n^{-1-2\alpha} 
    \dd x_1 \dots \dd x_4\\
\end{split}
\]
We can use Lemmata~\ref{lemma:C5} and~\ref{lemma:C6} to conclude that
the integral on the right-hand side with respect to $x_3$ is
\[
\begin{split}
\int_{I_n^1} \dots \dd x_3 &=
\int_{I_n^1} (\Iverson{\abs{x_2-x_4} > 3n^{-1}} + \Iverson{\abs{x_2-x_4}
\le 3n^{-1}} )\dots \dd x_3 \\
& \asymp  \Iverson{\gamma < 0} \abs{x_2-x_4}^{-1+2(\beta-\alpha)} \sum_{\rho \in
\brc{\alpha_+,\beta_-}} n^{2\rho} \abs{x_2-x_4}^{2\rho}\\
& + \Iverson{\gamma > 0} (x_2\wedge x_4)^{2\gamma}.
\end{split}
\]
where $\gamma = \alpha_- + \beta_+ - {\nicefrac1{2}}$. 
We can repeat this integration with respect to $x_1$ and combining these
we obtain that
\[
\begin{split}
\int_{I_n^4} \dots \dd x 
& \asymp  \Iverson{\gamma < 0} \!\!\!\!\!\!
\sum_{\rho_1, \rho_2 \in \brc{\alpha_+,\beta_-}} \!\!\!\!\!\! n^{2(\rho_1+\rho_2)}
\int_{I_n^2}\abs{x_2-x_4}^{-2+4(\beta-\alpha) +  
2(\rho_1+\rho_2)} \dd x\\
& + \Iverson{\gamma > 0} \int_{I_n^2} \abs{x_2-x_4}^{-2+4(\beta-\alpha)}(x_2\wedge x_4)^{4\gamma} \dd x
\end{split}
\]
 Supposing $\gamma > 0$. Then using the~\eqref{lemma:42:A1:B} from
 Lemma~\ref{lemma:42:A1} we have an upper estimate
 \[
 n^{4(\beta-\alpha)}\int_{I_n^2} \abs{x_2-x_4}^{-2+4(\beta-\alpha)}(x_2\wedge x_4)^{4\gamma}
 \dd x \lesssim 
 n^{4(\beta_++\alpha_-)}
 \]
 and we can easily deduce that the estimate holds also from below. When
 $\gamma < 0$, we can apply~\eqref{lemma:42:A1:A} from
 Lemma~\ref{lemma:42:A1} to conclude that
 \[
 n^{4(\beta-\alpha)}\int_{I_n^2} \dots \asymp \sum_{\rho_1, \rho_2 \in
 \brc{\alpha_+,\beta_-}} \!\!\!\!\!\! n^{(4(\beta-\alpha) +
 2(\rho_1+\rho_2)) \vee 1}
 \]
 and it is straightforward to verify that this is $\asymp
 n^{4(\beta_++\alpha_-) \vee 1}$. Combining both cases and taking the
 square root implies the claim.
\end{proof}

\begin{figure}
\begin{tikzpicture}[scale=2]
  \draw (4,0) -- (5,0) -- (5,1) -- (4,1) -- (4,0);
  \draw[dotted] (4,0) -- (5,1);
  \draw[dotted] (5,0) -- (4,1);
  \draw[dotted] (4.3333,0.6666) -- (5,1);
  \draw[dotted] (4,0) -- (4.3333,0.6666);
  \draw[dotted] (4.6666,0.3333) -- (4,0);
  \draw[dotted] (5,1) -- (4.6666,0.3333);
  \draw (4.3333,0.6666) -- (4.6666,0.8333) -- (4.6666,0.8333) -- (4.6666,1) --  (4,1); 
  \draw (4.3333,0.6666) -- (4.1666, 0.3333) -- (4,0.3333) --  (4,1);
  \draw (4.6666,0.3333) -- (4.3333, 0.1666) -- (4.3333,0.1666) -- (4.3333,0) --  (5,0);
  \draw (4.6666,0.3333) -- (4.8333, 0.6666) -- (4.8333,0.6666) -- (5,0.6666) --  (5,0);
  \draw [xshift=-5] (6,0) -- (7,0) -- (7,1) -- (6,1) -- (6,0);
  \draw[xshift=-5,dotted] (6,0) -- (7,1);
  \draw[xshift=-5,dotted] (7,0) -- (6,1);
  \draw[xshift=-5,dotted] (6.3333,0.6666) -- (7,1);
  \draw[xshift=-5,dotted] (6,0) -- (6.3333,0.6666);
  \draw[xshift=-5,dotted] (6.6666,0.3333) -- (6,0);
  \draw[xshift=-5,dotted] (7,1) -- (6.6666,0.3333);
  \draw [xshift=-5] (6.6666,0.8333) -- (6.6666,1) --  (7,1) 
   -- (6.6666,0.8333);
  \draw [xshift=-5] (6.8333, 0.6666) -- (7,0.6666) --  (7,1) 
   -- (6.8333, 0.6666);
  \draw [xshift=-5] (6.3333,0.1666) -- (6.3333,0) --  (6,0) 
   -- (6.3333,0.1666);
  \draw [xshift=-5] (6.1666,0.3333) -- (6,0.3333) --  (6,0) 
   -- (6.1666,0.3333);
\end{tikzpicture}
\caption{Supports of the kernels $E_3^{(-\alpha)}$ and
$E_4^{(-\alpha)}$} 
\end{figure}
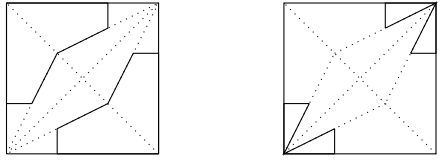

\begin{lem}
  \label{lemma:42:A5}
  For every $\alpha, \beta \in (-{\nicefrac1{2}}, {\nicefrac1{2}})$ we have that
  \[
  \norm{T_n(g_\beta)^{\nicefrac1{2}} \widetilde E_{3}^{(-\alpha)}
  T_n(g_\beta)^{\nicefrac1{2}}}_F
  \lesssim n^{2(\alpha_- + \beta_+) \vee {\nicefrac1{2}}}
  \]
  where the kernel of $E_3^{(-\alpha)}$ is
  $\mathcal S(E_2^{(-\alpha)} \Iverson{x \le \twothirds})$.
\end{lem}
\begin{proof}
We can use the rather trivial estimate 
\[
\mathcal S(E_2^{(-\alpha)} \Iverson{x \le \twothirds}) \lesssim n^{2\alpha_+}
\]
which follows since $E_2^{(-\alpha)}(x,y) \lesssim n^{2\alpha_+}$ in
when $x \le \frac23$. 
This implies that
  \[
  \begin{split}
  &\norm{T_n(g_\beta)^{\nicefrac1{2}} \widetilde E_{3}^{(-\alpha)}
  T_n(g_\beta)^{\nicefrac1{2}}}^2_F\\
  &\lesssim n^{4(\beta-\alpha)}n^{4\alpha_+} \int_{I_n^4}
  \abs{x_1-x_2}^{-1+2\beta}_n \abs{x_3-x_4}_n^{-1+2\beta} \dd x_1\dots \dd
  x_4 \\
  &\lesssim
     n^{4(\beta+\alpha_-)}\Big(\int_{I_n^2} \abs{x_1 -
     x_2}_n^{-1+2\beta} \dd x_1 \dd x_2\Big)^2
  \end{split}
  \]
The integral on the right-hand side is a special case
of~\eqref{lemma:42:A1:B} in Lemma~\ref{lemma:42:A1} when $\alpha = -1$.
Therefore, 
  \[
  \begin{split}
  &\norm{T_n(g_\beta)^{\nicefrac1{2}} \widetilde E_{3}^{(-\alpha)}
  T_n(g_\beta)^{\nicefrac1{2}}}^2_F\\
  &\lesssim
     n^{4(\beta+\alpha_-)}n^{4\beta_-} = 
     n^{4(\beta_++\alpha_-)}
  \end{split}
  \]
  and the claim follows.
\end{proof}
\begin{lem}
  \label{lemma:42:A6}
  For every $\alpha, \beta \in (-{\nicefrac1{2}}, {\nicefrac1{2}})$ we have that
  \[
  \norm{T_n(g_\beta)^{\nicefrac1{2}} \widetilde E_{4}^{(-\alpha)}
  T_n(g_\beta)^{\nicefrac1{2}}}_F
  \lesssim n^{2(\alpha_- + \beta_+)} 
  \]
  where the kernel of $E_4^{(-\alpha)}$ is
  $\mathcal S(E_2^{(-\alpha)} \Iverson{x > \twothirds})$.
\end{lem}
\begin{proof}
We can use a similar estimate as in Lemma~\ref{lemma:42:A5} but this
time we cannot estimate the indicator function with a constant. Thus, we
use an estimate
\[
\mathcal S(E_2^{(-\alpha)} \Iverson{x > \twothirds}) \lesssim n^{\alpha_+}
\mathcal S(\Iverson{x > \twothirds} \widetilde x^{-1-\alpha})
\]
which follows since $E_2^{(-\alpha)} (x,y) \lesssim n^{\alpha_+}
\abs{x-y}^{-1-\alpha}$ in when $x > \frac23$ and moreover we can estimate $\abs{x-y}^{-1-\alpha}
\lesssim \widetilde x^{-1-\alpha}$ given $x > \frac23$. We can divide the
support into the four pieces and denote them by $J_{4,1}, \dots,
J_{4,4}$ and then express the right-hand side as a sum
\[
\begin{split}
&\mathcal S(\Iverson{x > \twothirds} \widetilde x^{-1-\alpha})\\
&= \widetilde x^{-1-\alpha} \Iverson{J_{4,1}}
+ \widetilde y^{-1-\alpha} \Iverson{J_{4,2}}
+ x^{-1-\alpha} \Iverson{J_{4,3}}
+ y^{-1-\alpha} \Iverson{J_{4,4}}\\
&=: I_1 + \dots + I_4
\end{split}
\]
Therefore, the squared norm can be estimated as
\[
\begin{split}
  &\norm{T_n(g_\beta)^{\nicefrac1{2}} \widetilde E_{4}^{(-\alpha)}
  T_n(g_\beta)^{\nicefrac1{2}}}_F^2 n^{-4(\beta-\alpha ) - 2\alpha_+}\\
  &\lesssim
  \sum_{j,k=1}^4 \int_{I_n^4} \abs{x_1-x_2}^{-1+2\beta} I_j(x_2,x_3)
  \abs{x_3-x_4}^{-1+2\beta} I_k(x_4,x_1) \dd x\\
  &= 4
  \sum_{j=1}^4 \int_{I_n^4} \abs{x_1-x_2}^{-1+2\beta} I_3(x_2,x_3)
  \abs{x_3-x_4}^{-1+2\beta} I_j(x_4,x_1) \dd x 
\end{split}
\]
where the last identity follows by using the symmetries
$(x_j)_j \leftrightarrow (\widetilde x_j)_j$ and 
$(x_j)_j \leftrightarrow (x_{5-j})_j$. The terms in the sum are
essentially of two types, which we can call \emph{evenly} (when $j\in
\brc{1,3}$) and \emph{unevenly bound} (when $j \in \brc{2,4}$).

The evenly bound terms are easier, since we don't need the indicators
any more and we can use estimates
$I_3(x,y) \lesssim x^{-1-\alpha}$ and $I_1(x,y) \lesssim \widetilde
x^{-1-\alpha}$. This implies that for the first evenly bound case ($j
= 3$) we have
\[
\begin{split}
&\int_{I_n^4} \abs{x_1-x_2}^{-1+2\beta} I_3(x_2,x_3)
  \abs{x_3-x_4}^{-1+2\beta} I_3(x_4,x_1) \dd x \\
&\lesssim \Big(\int_{I_n^2} \abs{x_1-x_2}^{-1+2\beta} x_2^{-1-\alpha} \dd x_1
\dd x_2\Big)^2 \asymp n^{4\beta_- + 2\alpha_+}
\end{split}
\]
where the last estimate follows directly from~\eqref{lemma:42:A1:B} in
Lemma~\ref{lemma:42:A1}. The second evenly bound case needs one extra
application of symmetry $(x_3,x_4) \leftrightarrow (\widetilde x_1,
\widetilde x_2)$ after the four-dimensional integral has been
split to a product of two two-dimensional integrals. Thus, also the
second evenly bound case has exactly the estimate, namely
\[
\begin{split}
&\int_{I_n^4} \abs{x_1-x_2}^{-1+2\beta} I_3(x_2,x_3)
  \abs{x_3-x_4}^{-1+2\beta} I_1(x_4,x_1) \dd x 
\lesssim n^{4\beta_- + 2\alpha_+}
\end{split}
\]

For estimating the unevenly bound cases ($j\in \brc{2,4}$) we have to
take the indicator functions into account. In both cases we can first
integrate with respect to $x_3$. The function depending on $x_3$ in both
cases is $\Iverson{(x_2,x_3) \in J_{4,3}} \abs{x_3-x_4}^{-1+2\beta}$. This can be easily
estimated 
\[
\int_{I_n^1} \Iverson{(x_2,x_3) \in J_{4,3}} \abs{x_3-x_4}^{-1+2\beta}
\dd x_3 \lesssim n^{2\beta_-}
\]
Next we integrate with respect to $x_4$. The reminding part depending on
$x_4$ is just the indicator function $\Iverson{(x_4,x_1) \in J_{4,j}}$. In the
first unevenly bound case ($j = 4$) we can estimate this by 
\[
\int_{I_n^1}\Iverson{(x_4,x_1) \in J_{4,4}} \dd x_4 \le \int_{I_n^1}
\Iverson{x_4 \le {{\puo}x_1}} \dd x_4 \lesssim x_1
\]
This means that we have an upper estimate
\[
\begin{split}
&\int_{I_n^4} \abs{x_1-x_2}^{-1+2\beta} I_3(x_2,x_3)
  \abs{x_3-x_4}^{-1+2\beta} I_4(x_4,x_1) \dd x \\
&\lesssim n^{2\beta_-}\int_{I_n^2} x_1^{-\alpha} x_2^{-1-\alpha}
\abs{x_1-x_2}^{-1+2\beta} \dd x_1 \dd x_2.
\end{split}
\]
The singularity at $x_1 = 0$ is integrable, so we first integrate with
respect to $x_1$. We can split the integration into two parts
\[
\begin{split}
\int_{I_n^1} x_1^{-\alpha} \abs{x_1-x_2}^{-1+2\beta} \dd x_1
&= \int_{I_n^{-1}}\big( \Iverson{x_1 \le \puo x_2} + \Iverson{x_1 > \puo
x_2} \big)
\dots\\
&\lesssim x_2^{2\beta - \alpha} + n^{2\beta_-} x_2^{-\alpha} \lesssim
n^{2\beta_-} x_2^{-\alpha}
\end{split}
\]
and hence
\[
\begin{split}
n^{2\beta_-}\int_{I_n^2} x_1^{-\alpha} x_2^{-1-\alpha}
\abs{x_1-x_2}^{-1+2\beta} \dd x_1 \dd x_2
&\lesssim n^{4\beta_-} \int_{I_n^1} x_2^{-1-2\alpha} \dd x_2 \\
&\asymp n^{4\beta_- + 2\alpha_+}
\end{split}
\]

Similarly, the second unevenly bound case can be estimated to give
\[
\int_{I_n^1}\Iverson{(x_4,x_1) \in J_{4,2}} \dd x_4 
\lesssim \widetilde x_1
\]
This means that 
\[
\begin{split}
&\int_{I_n^4} \abs{x_1-x_2}^{-1+2\beta} I_3(x_2,x_3)
  \abs{x_3-x_4}^{-1+2\beta} I_2(x_4,x_1) \dd x \\
&\lesssim n^{2\beta_-}\int_{I_n^2} \widetilde x_1^{-\alpha} x_2^{-1-\alpha}
\abs{x_1-x_2}^{-1+2\beta} \dd x_1 \dd x_2.
\end{split}
\]
By symmetry, the integral with respect to $x_1$ has an estimate
\[
\begin{split}
\int_{I_n^1} \widetilde x_1^{-\alpha} \abs{x_1-x_2}^{-1+2\beta} \dd x_1
\lesssim n^{2\beta_-} \widetilde x_2^{-\alpha}
\end{split}
\]
which means that the singularity is split in two parts and we obtain
\[
\begin{split}
n^{2\beta_-}\int_{I_n^2} \widetilde x_1^{-\alpha} x_2^{-1-\alpha}
\abs{x_1-x_2}^{-1+2\beta} \dd x_1 \dd x_2
&\lesssim n^{4\beta_-} \int_{I_n^1} x_2^{-1-\alpha} \widetilde
x_2^{-\alpha} \dd x_2 \\
&\asymp n^{4\beta_- + \alpha_+}
\end{split}
\]
Therefore, when we combine all the previous estimates we obtain the
claimed estimate for the squared norm
\[
\begin{split}
  &\norm{T_n(g_\beta)^{\nicefrac1{2}} \widetilde E_{4}^{(-\alpha)}
  T_n(g_\beta)^{\nicefrac1{2}}}_F^2 
  \lesssim n^{4\beta_- + 2\alpha_+}
n^{4(\beta-\alpha ) + 2\alpha_+}
= n^{4(\beta_++\alpha_-)}.
\end{split}
\]

\end{proof}
\begin{lem}
  \label{lemma:44:B}
    The estimate~\eqref{eq:p:B:L3:2} holds for every $0 \ne \alpha \in
    (-\puo,\puo)$. 
\end{lem}
\begin{proof}
  Suppose $y \notin \mathcal S(I_d(x))$. Then  we have $k_\alpha(t,y) \asymp
  k_\alpha(x,y)$ for every $t \in \mathcal S(I_d(x))$ uniformly in $t$.

  When $y \in \mathcal S(I_d(x))$  we have that $k_\alpha(t,y) \asymp
  k_\alpha(x,x)$ uniformly in $t \in \mathcal S(I_d(x))$. Therefore,
  $k_\alpha(t,y) \asymp k_\alpha(x,y)$ uniformly for every $t \in
  \mathcal S(I_d(x))$. This implies that
  \begin{equation}
    \label{lemma:44:B:p:1}
    \begin{split}
    &\int_{n^{-1}}^{1-n^{-1}} \abs{x-t}_n^{-1-2\alpha}
    k_\alpha(t,y) \Iverson{ t \in \mathcal S(I_d(x))} \dd t\\
    & \asymp k_\alpha(x,y) 
    \int_{n^{-1}}^{1-n^{-1}} \abs{x-t}_n^{-1-2\alpha} \Iverson{ t \in
    \mathcal S(I_d(x))} \dd t\\
    & \asymp k_\alpha(x,y) 
    \big(n^{2\alpha}\Iverson{\alpha > 0} + \Iverson{\alpha < 0}\big)
    \end{split}
  \end{equation}
  where the last estimate follows by direct integration. This implies
  the claim.
\end{proof}
\begin{lem}
  \label{lemma:44:C}
    The estimate~\eqref{eq:p:B:L3:3} holds for every $0 \ne \alpha \in
    (-\puo,\puo)$. 
\end{lem}
\begin{proof}
  In order to obtain the estimate~\eqref{eq:p:B:L3:3}, we divide the
  integration set $\brc{t \notin \mathcal S(I_d(x)}$ into \emph{lower}
  and \emph{upper} parts, where an element $t \notin \mathcal S(I_d(x))$
  belongs to $t \in \text{lower}$ when $n^{-1} < t < x$ and $t \in
  \text{upper}$ when
  $x < t < 1-n^{-1}$. Therefore, we can write
  \[
  \int_{n^{-1}}^{1-n^{-1}} \dots \Iverson{t \notin \mathcal S(I_d(x))}
  \dd t = 
  \int_{\text{lower}} + \int_{\text{upper}} \dots \, \dd t
  \]

  We can exploit the symmetry $k_\alpha(x,y) = k_\alpha(\widetilde x,\widetilde y)$
  and $\abs{t-x} = \abs{\widetilde t - \widetilde x}$ together with
  change of variables that to reduce showing that the estimate
  \begin{equation}
    \label{lemma:44:C:p:1}
    \int_{\text{lower}} \mathcal S(E_2^{\alpha})(x,t) k_\alpha(t,y) \dd t  \lesssim n^{2\alpha_+} k_\alpha(x,y).
  \end{equation}
  holds for every $x, y$ for the integral over the $\text{lower}$ interval.

  This in turn is obtained by showing the estimate by assuming in
  addition that $(x,y) \in J_1, J_2, J_3$ or $(x, y) \in J_4$ where
  \begin{equation}
    \label{lemma:44:C:p:4}
    \begin{split}
      J_1 & := \set{(x,y)}{x < \mbox{$\frac23$}, y \in \mathcal S(I_d(x)) \text{ or } x
      <\mbox{$\frac23$} \wedge y}\\
      J_2 & := \set{(x,y)}{\mbox{$\frac23$} \le x < 1, y \in \mathcal
      S(I_d(x)) \text{ or }
      \mbox{$\frac23$} \le x < y}\\
      J_3 & := \set{(x,y) \in \mathcal S(I_b)}{y < x <
      \mbox{$\frac23$}}\\
      J_4 & := \set{(x,y) \in \mathcal S(I_b)}{x > y \vee
      \mbox{$\frac23$}}
    \end{split}
  \end{equation}
\begin{figure}
\begin{tikzpicture}[scale=1.7]
  \draw [xshift=10] (0,0) -- (1,0) -- (1,1) -- (0,1) -- (0,0);
  \draw[xshift=10,dotted] (0,0) -- (1,1);
  \draw[xshift=10,dotted] (1,0) -- (0,1);
  \draw[xshift=10,dotted] (0.3333,0.6666) -- (1,1);
  \draw[xshift=10,dotted] (0,0) -- (0.3333,0.6666);
  \draw[xshift=10,dotted] (0.6666,0.3333) -- (0,0);
  \draw[xshift=10,dotted] (1,1) -- (0.6666,0.3333);
  \draw [xshift=10] (0.3333,0.6666) -- (0,0) -- (0.6666,0) --
  (0.6666,0.8333) -- (0.3333,0.6666);
  \draw [xshift=5] (2,0) -- (3,0) -- (3,1) -- (2,1) -- (2,0);
  \draw[xshift=5,dotted] (2,0) -- (3,1);
  \draw[xshift=5,dotted] (3,0) -- (2,1);
  \draw[xshift=5,dotted] (2.3333,0.6666) -- (3,1);
  \draw[xshift=5,dotted] (2,0) -- (2.3333,0.6666);
  \draw[xshift=5,dotted] (2.6666,0.3333) -- (2,0);
  \draw[xshift=5,dotted] (3,1) -- (2.6666,0.3333);
  \draw [xshift=5] (2.6666,0.8333) -- (2.6666,0) -- (3,0) -- (3,1)
   -- (2.6666,0.8333);
  \draw (4,0) -- (5,0) -- (5,1) -- (4,1) -- (4,0);
  \draw[dotted] (4,0) -- (5,1);
  \draw[dotted] (5,0) -- (4,1);
  \draw[dotted] (4.3333,0.6666) -- (5,1);
  \draw[dotted] (4,0) -- (4.3333,0.6666);
  \draw[dotted] (4.6666,0.3333) -- (4,0);
  \draw[dotted] (5,1) -- (4.6666,0.3333);
  \draw (4.6666,0.8333) -- (4.6666,1) --  (4,1) 
   -- (4,0) -- (4.3333,0.6666) -- (4.6666,0.8333);
  \draw [xshift=-5] (6,0) -- (7,0) -- (7,1) -- (6,1) -- (6,0);
  \draw[xshift=-5,dotted] (6,0) -- (7,1);
  \draw[xshift=-5,dotted] (7,0) -- (6,1);
  \draw[xshift=-5,dotted] (6.3333,0.6666) -- (7,1);
  \draw[xshift=-5,dotted] (6,0) -- (6.3333,0.6666);
  \draw[xshift=-5,dotted] (6.6666,0.3333) -- (6,0);
  \draw[xshift=-5,dotted] (7,1) -- (6.6666,0.3333);
  \draw [xshift=-5] (6.6666,0.8333) -- (6.6666,1) --  (7,1) 
   -- (6.6666,0.8333);
\end{tikzpicture}
\caption{Illustration of indicators of $J_j$}
\end{figure}
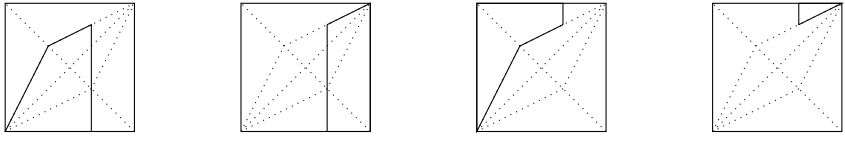
  Since $\brc{J_1,\dots,J_4}$ is a partition of the unit square, these
  together yield the claim. These estimates follow from
  Lemmata~\ref{lemma:44:C1},~\ref{lemma:44:C2},~\ref{lemma:44:C3}
  and~\ref{lemma:44:C4}, respectively.
\end{proof}
\begin{lem}
  \label{lemma:44:C1}
    Suppose $(x,y) \in J_1$ where $J_1$ is defined as
    in~\eqref{lemma:44:C:p:4}. Then the estimate~\eqref{lemma:44:C:p:1} holds
    for every $\alpha \in (-\puo,\puo) \setminus \brc{0}$. 
\end{lem}
\begin{proof}
  In this case, we notice that we can estimate the integral on the
  left-hand side of~\eqref{lemma:44:C:p:1} by
  \[
  \int_{\text{lower}} \dots \dd t \asymp x^{-1-\alpha}
  \int_{n^{-1}}^{x/2} \big(t^{-\alpha - 2\abs{\alpha}} y^{-1+2\abs{\alpha}} +
  t^{-\alpha} \widetilde y^{-2\abs\alpha}\big) \dd t
  \]
  In order to estimate this we divide this into three parts $\brc{\alpha
  < 0}$, $\brc{0 < \alpha < \frac13}$ and $\brc{\frac13<\alpha<\puo}$.
  Therefore, by direct integration we obtain an estimate
  \[
  \begin{split}
  \int_{\text{lower}} \dots \dd t &\asymp \Iverson{\alpha < 0} y^{-1+2\abs{\alpha}}
  + \Iverson{0 < \alpha < \mbox{$\frac13$}}(x^{-4\abs\alpha}
  y^{-1+2\abs{\alpha}}) \\
  &+ \Iverson{\mbox{$\frac13$} <\alpha<\puo}x^{-1-\alpha}n^{3\alpha
  -1}y^{-1+2\abs\alpha}
  +
  x^{-2\alpha}\widetilde y^{-2\abs\alpha}
  \end{split}
  \]
  In this region $k_\alpha(x,y) \asymp
  x^{-2\abs\alpha}y^{-1+2\abs\alpha}+ \widetilde y^{-2\abs\alpha}$ and
  the claim follows, since every term can be bounded from above by
  $n^{2\alpha_+}k_\alpha(x,y)$.
\end{proof}
\begin{lem}
  \label{lemma:44:C2}
    Suppose $(x,y) \in J_2$ where $J_2$ is defined as
    in~\eqref{lemma:44:C:p:4}. Then the estimate~\eqref{lemma:44:C:p:1} holds
    for every $\alpha \in (-\puo,\puo) \setminus \brc{0}$. 
\end{lem}
\begin{proof}
  In this case we have an estimate $k_\alpha(x,y) \asymp \widetilde
  x^{-1+2\abs\alpha}\widetilde y^{-2\abs\alpha}$. Moreover, the integral
  can be estimated as
  \[
  \begin{split}
  \int_{\text{lower}} \dots \dd t &\asymp 
  \widetilde x^{-\alpha} \int_{n^{-1}}^{2x-1} (x-t)^{-1-\alpha}t^{-\alpha-2\abs\alpha} \dd t
  \\
  &+ \widetilde x^{-\alpha} \widetilde y^{-2\abs\alpha}
  \int_{n^{-1}}^{2x-1} (x-t)^{-1-\alpha}t^{-\alpha}\widetilde t^{-1+ 2\abs\alpha} \dd t
  \end{split}
  \]
  Let's denote the right-hand side as $I_1 + I_2$. When $\alpha < 0$,
  the $I_1$ can be easily estimated, since then
  \[
  \widetilde x^{-\alpha} \int_{n^{-1}}^{x/2} \!\!\!+ \int_{x/2}^{2x-1}
  \!\!\!
  \dots\,
  \asymp 
  \widetilde x^{\abs\alpha} \Big( 1 + \int_{x/2}^{2x-1} \!\!\!
  (x-t)^{-1+\abs\alpha} \dd t\Big) \lesssim k_\alpha(x,y)
  \]
  When $\alpha < 0$, the part $I_2$ can be estimated from above as
  \[
  \begin{split}
  I_2 &\lesssim \widetilde x^{\abs\alpha} \widetilde y^{-2\abs\alpha}
   \Big(\int_{n^{-1}}^{x/2} t^{-\alpha} \dd t + \int_{x/2}^{2x-1}
   (x-t)^{-2+3\abs\alpha} \dd t\Big) \\
   &\lesssim \Iverson{\alpha < -\mbox{$\frac13$}} \widetilde
   x^{\abs\alpha}\widetilde y^{-2\abs\alpha} +
   \Iverson{-\mbox{$\frac13$}
   < \alpha < 0} \widetilde x^{-1+4\abs\alpha} y^{-2\abs\alpha}\\
   & \lesssim k_\alpha(x,y)
  \end{split}
  \]
  When $\alpha > 0$, the part $I_2$ can be estimated as
  \[
  \begin{split}
  I_2 &\lesssim \widetilde x^{-\abs\alpha} \widetilde y^{-2\abs\alpha}
   \Big(\int_{n^{-1}}^{x/2} t^{-\alpha} \dd t + \int_{x/2}^{2x-1}
   (x-t)^{-2+\abs\alpha} \dd t\Big) \\
   &\lesssim \widetilde
   x^{-1}\widetilde y^{-2\abs\alpha} 
    \lesssim n^{2\alpha} k_\alpha(x,y)
  \end{split}
  \]
  When $\alpha > 0$, the part $I_1$ estimate divides according to
  whether $\alpha > \frac13$, $\alpha < \frac13$ or $\alpha =
  \frac13$. In two former cases $I_1 \lesssim k_\alpha(x,y)$ and in the
  last case $I_1 \lesssim \log n k_\alpha(x,y)$ which are all majorized
  by $n^{2\alpha} k_\alpha(x,y)$. In all cases we have an estimate
  \[
  \begin{split}
  I_1 &\asymp \widetilde x^{-2\alpha} + x^{-\abs\alpha}
  \int_{n^{-1}}^{x/2} t^{-3\alpha} \dd t\\
   &\lesssim k_\alpha(x,y) \big( \log n \Iverson{0 < \alpha \le \mbox{$\frac13$}}
   + \Iverson{\alpha > \mbox{$\frac13$} } n^{3\alpha -1} \widetilde x^{1
   - 3\alpha}\big)\\
   &\lesssim n^{2\alpha} k_\alpha(x,y)
  \end{split}
  \]
\end{proof}
\begin{lem}
  \label{lemma:44:C3}
    Suppose $(x,y) \in J_3$ where $J_3$ is defined as
    in~\eqref{lemma:44:C:p:4}. Then the estimate~\eqref{lemma:44:C:p:1} holds
    for every $\alpha \in (-\puo,\puo) \setminus \brc{0,
    \mbox{$\frac13$}}$. When $\alpha \in \brc{0,\mbox{$\frac13$}}$ the
    estimate holds with logarithmic correction.
\end{lem}
\begin{proof}
  In this case we have an estimate $k_\alpha(x,y) \asymp 
  x^{-1+2\abs\alpha} y^{-2\abs\alpha}$. In this case, we divide the
  integration interval into two parts 
  \[
  \int_{\text{lower}} \dots \dd t = 
  \int_{n^{-1}}^{y} +
  \int_y^{x/2} \dots \dd t =: I_1 + I_2.
  \]
  The latter part $I_2$ is easier to estimate since
  \[
  \begin{split}
  I_2 &\asymp
  x^{-1-\alpha} y^{-2\abs\alpha} \int_y^{x/2} t^{-1 - \alpha
  +2\abs\alpha} \dd t
  \asymp
  k_\alpha(x,y) x^{-2\alpha} \\
  &\lesssim n^{2\alpha_+} k_\alpha(x,y)
  \end{split}
  \]
  The former part $I_1$ needs bit more. We can estimate that
  \[
  I_1 \asymp x^{-1-\alpha}y^{-1+2\abs\alpha} \int_{n^{-1}}^y t^{-\alpha
  - 2\abs\alpha}
  \]
  When $\alpha < 0$, we therefore have
  \[
  I_1 \asymp  x^{-1+\abs\alpha}y^{\abs\alpha} \asymp k_\alpha(x,y)
  x^{-\abs\alpha}y^{3\abs\alpha} \lesssim k_\alpha(x,y)
  \]
  When $\alpha > 0$, then we have two cases $\alpha <
  \mbox{$\frac13$}$ or $\mbox{$\frac13$} < \alpha < \puo$. In the former we
  estimate
  \[
  I_1 \asymp k_\alpha(x,y) x^{-3\alpha}y^{\alpha} \lesssim k_\alpha(x,y)
  x^{-2\alpha} \lesssim n^{2\alpha} k_\alpha(x,y)
  \]
  and in the latter
  \[
  I_1 \asymp k_\alpha(x,y) x^{-3\alpha}y^{-1 + 4\alpha} n^{3\alpha - 1} \lesssim k_\alpha(x,y)
  x^{-1+\alpha} n^{3\alpha - 1} \lesssim n^{2\alpha} k_\alpha(x,y)
  \]
\end{proof}
\begin{lem}
  \label{lemma:44:C4}
    Suppose $(x,y) \in J_4$ where $J_4$ is defined as
    in~\eqref{lemma:44:C:p:4} and suppose in addition that $y \le
    \mbox{$\frac16$}$. Then the estimate~\eqref{lemma:44:C:p:1} holds
    for every $\alpha \in (-\puo,\puo) \setminus \brc{0}$. 
\end{lem}
\begin{proof}
In this case $k_\alpha(x,y) \asymp y^{-2\abs\alpha} + \widetilde
x^{-2\abs\alpha}$ and we divide the integration interval into three parts 
  \[
  \int_{\text{lower}} \dots \dd t = 
  \int_{n^{-1}}^{y} +
  \int_y^{1/6} + \int_{1/6}^{2x-1} \dots \dd t =: I_1 + I_2 + I_3.
  \]
The first integral can be estimated as
\[
\begin{split}
I_1 &\asymp \widetilde x^{-\alpha} \int_{n^{-1}}^y t^{\alpha -
2\abs\alpha} y^{-1+2\abs\alpha} + t^{-\alpha} \dd t\\
&\asymp \widetilde x^{-\alpha} \big( \Iverson{\alpha < \mbox{$\frac13$}}
y^{-\alpha} + \Iverson{\alpha > \mbox{$\frac13$}} n^{3\alpha -1}
y^{-1+2\alpha} + y^{1-\alpha} \big) \\
&\lesssim k_\alpha(x,y) \big(\Iverson{\alpha < \mbox{$\frac13$}} +
\Iverson{\alpha > \mbox{$\frac13$}} n^{\alpha}\big)
\lesssim n^{\alpha_+} k_\alpha(x,y)
\end{split}
\]
The estimation of the second integral is easier, since
\[
\begin{split}
I_2 &\asymp \widetilde x^{-\alpha} \int_{y}^{\mbox{$\frac16$}} \big(t^{-1-\alpha +
2\abs\alpha} y^{-2\abs\alpha} + t^{-\alpha} \big)\dd t\\
\end{split}
\]
Now the antiderivative functions are increasing functions for every
$\alpha$ and we have a constant upper integration upper bound and therefore,
\[
\begin{split}
I_2 &\asymp \widetilde x^{\mp\abs\alpha} y^{-2\abs\alpha}  \lesssim
n^{\alpha_+} k_\alpha(x,y)
\end{split}
\]
In the last integral $I_3$ we need to take into account the terms of form
$(x-t)^\gamma$ but not the terms of form $t^\gamma$ and so
\[
\begin{split}
I_3 &\asymp \widetilde x^{-\alpha} \int_{\mbox{$\frac16$}}^{2x-1}
(x-t)^{-1-\alpha} \big(y^{-2\abs\alpha} + \widetilde t^{-2\abs\alpha} \big)\dd t\\
\end{split}
\]
When $\alpha < 0$, we can therefore estimate that
\[
\begin{split}
I_3 &\asymp \widetilde x^{\abs\alpha} \Big( y^{-2\abs\alpha} + \int_{\mbox{$\frac16$}}^{2x-1}
(x-t)^{-1+\abs\alpha} \widetilde t^{-2\abs\alpha} \dd t \Big)\\
& \lesssim \widetilde x^{\abs\alpha} \Big( y^{-2\abs\alpha} + \widetilde
x^{-\abs\alpha} \Big) \lesssim k_\alpha(x,y)
\end{split}
\]
When $\alpha > 0$, we have
\[
\begin{split}
I_3 &\asymp \widetilde x^{-2\abs\alpha}  y^{-2\abs\alpha} + \widetilde
x^{-3\abs\alpha} \int_{\mbox{$\frac16$}}^{2x-1}
(x-t)^{-1+\abs\alpha} \dd t \\
& \lesssim n^{2\alpha} k_\alpha(x,y) + \widetilde x^{-4\abs\alpha}
 \lesssim n^{2\alpha} k_\alpha(x,y)
\end{split}
\]
since $\widetilde t^{-2\abs\alpha}$ behaves like $\widetilde
x^{-2\abs\alpha}$ when $t$ is near $2x-1$. Combining the estimates, we
obtain  the claim.
\end{proof}
\begin{lem}
  \label{lemma:44:C5}
    Suppose $(x,y) \in J_4$ where $J_4$ is defined as
    in~\eqref{lemma:44:C:p:4} suppose in addition that $y >
    \mbox{$\frac16$}$. Then the estimate~\eqref{lemma:44:C:p:1} holds
    for every $\alpha \in (-\puo,\puo) \setminus \brc{0}$. 
\end{lem}
\begin{proof}
In this case $k_\alpha(x,y) \asymp \widetilde y^{-1+2\abs\alpha} \widetilde
x^{-2\abs\alpha}$ and we divide the integration interval into three parts 
  \[
  \int_{\text{lower}} \dots \dd t = 
  \int_{n^{-1}}^{1/6} +
  \int_{1/6}^y + \int_{y}^{2x-1} \dots \dd t =: I_1 + I_2 + I_3.
  \]
The first integral can be estimated as
\[
\begin{split}
I_1 &\asymp \widetilde x^{-\alpha} \big( \Iverson{\alpha < \mbox{$\frac13$}}
 + \Iverson{\alpha > \mbox{$\frac13$}} n^{3\alpha -1}
y^{-1+2\alpha} + \widetilde y^{-2\abs\alpha} \big) \\
\end{split}
\]
Therefore, when $\alpha < 0$ we have $\widetilde x \le \widetilde y$ and
thus $I_1 \lesssim \widetilde x^{-\abs{\alpha}} \lesssim
k_\alpha(x,y)$. When $0 < \alpha < \mbox{$\frac13$}$, we can similarly
estimate that $I_1 \lesssim k_\alpha(x,y) \widetilde y^{1- 3\abs\alpha}
\lesssim k_\alpha(x,y)$. The leading order singularity for $I_1$ comes
when $\alpha > \mbox{$\frac13$}$, where we have $I_1 \asymp 
k_\alpha(x,y) \big(\widetilde x^{\alpha}
n^{3\alpha - 1} + \widetilde x^{\alpha} \widetilde y^{1-4\alpha}\big) \lesssim
k_\alpha(x,y) n^\alpha$.

The second integral can be first estimated as
\[
\begin{split}
I_2 &\asymp \widetilde x^{-\alpha} \int_{\mbox{$\frac16$}}^y
(x-t)^{-1-\alpha}\big( 1 +
\widetilde y^{-2\abs\alpha} \widetilde t^{-1 + 2\abs\alpha} \big)\dd t\\
&
\lesssim \widetilde x^{-\alpha}  \widetilde y^{-1} \big(\Iverson{\alpha <
0} + \Iverson{\alpha > 0} (x-y)^{-\abs{\alpha}}\big)\\
&\lesssim k_\alpha(x,y) n^{2\alpha_+}
\end{split}
\]
where we also used the estimate $\widetilde t \ge \widetilde y$ and when
$\alpha > 0$ we estimate $(x-y)^{-\alpha} \lesssim \widetilde
x^{-\alpha}$.

When $\alpha < 0$ the last integral $I_3$ can be estimated
\[
\begin{split}
I_3 &\asymp k_\alpha(x,y) \widetilde x^{3\abs\alpha} \int_y^{2x-1}
(x-t)^{-1+\abs\alpha} \widetilde t^{-2\abs\alpha} \dd t\\
&\lesssim k_\alpha(x,y) \widetilde x^{2\alpha} \lesssim k_\alpha(x,y)
\end{split}
\]
and when $\alpha > 0$ we estimate
\[
\begin{split}
I_3 &\asymp k_\alpha(x,y) \widetilde x^{\alpha} \int_y^{2x-1}
(x-t)^{-1-\alpha} \widetilde t^{-2\alpha} \dd t\\
&\lesssim k_\alpha(x,y) \widetilde x^{-2\alpha} \lesssim k_\alpha(x,y)
n^{2\alpha}
\end{split}
\]
Combining all the estimates, we obtain the claim.
\end{proof}

\begin{lem}
  \label{lemma:C5}
   When $y > x + \frac3n$ and $\gamma := \alpha_- + \beta_+ - {\nicefrac1{2}}$ we have
  \[
  \begin{split}
  &\int_{n^{-1}}^{(x+ y)/2}
  \abs{x-t}_n^{-1-2\alpha}\abs{y-t}_n^{-1+2\beta} \dd t \\
  &\asymp 
  \Iverson{\gamma < 0} \abs{y-x}_n^{-1+2(\beta-\alpha)}
  n^{2\beta_-}\abs{y-x}_n^{2\beta_-}
  + \Iverson{\gamma > 0} x^{2\gamma}
  \end{split}
  \]
\end{lem}
\begin{proof}
  This is \eqref{lemma:42:A1:C} from Lemma~\ref{lemma:42:A1}.
\end{proof}
\begin{lem}
  \label{lemma:C6}
   When $y > x + \frac3n$ and $\gamma := \alpha_- + \beta_+ - {\nicefrac1{2}}$ we have
  \[
  \begin{split}
  &\int_{(x+ y)/2}^{1-n^{-1}}
  \abs{x-t}_n^{-1-2\alpha}\abs{y-t}_n^{-1+2\beta} \dd t \\
  &\asymp 
  \Iverson{\gamma < 0} \abs{y-x}_n^{-1+2(\beta-\alpha)}
  n^{2\alpha_+}\abs{y-x}_n^{2\alpha_+}
  + \Iverson{\gamma > 0} x^{2\gamma}
  \end{split}
  \]
\end{lem}
\begin{proof}
This follows from Lemma~\ref{lemma:C5} by denoting
$y' := \widetilde x$, $x' := \widetilde y$, $\alpha' := -\beta$
and $\beta' := -\alpha$
and using  change of variables $t' = 1 - t$.
\end{proof}

\section{Proofs of auxiliary results in Section~\ref{factorisation}}
\label{PFC}

In this section we prove the technical results that were mentioned in
Section~\ref{factorisation}. These augment the results of Rambour and
Seghier~\cite{RambourSeghier, RambourSeghier2} to our setting.
\begin{proof}[Proof of Lemma~\ref{lemma:C:L1}]
The proof of this was sketched already before the claim of
Lemma~\ref{lemma:C:L1}, but let's provide some extra details. First let
$u \in H^2(\D)$ be the unique solution of the Dirichlet problem for
Laplace equation
\[
\begin{cases}
 \lapl u = 0, & \text{in } \D,\\ 
 u|_{\partial \D} = \puo \nu
\end{cases}
\]
where $\nu := \log f$.
Let us define the analytic function $F = u + iv$. It is well known that
the one harmonic conjugate $v$ is obtained as a solution of
\[
\begin{cases}
 \lapl v = 0, & \text{in } \D,\\ 
 v|_{\partial \D} = \puo \hilbert \nu.
\end{cases}
\]
All the others are of form $v + C$ for some constant $C \in \C$ since
the Hilbert transform on the torus maps constants to zero.
We can explicitly express the functions $u$ and $v$ and $F$ in terms of
the Fourier coefficients of $\nu := \log f$ on the boundary of the disk
$\D$, namely
\[
\begin{split}
u(e^{it}) &= \sum_{k \in \Z} \puokki \widehat \nu(k) e^{it k}\\
v(e^{it}) &= \sum_{k \in \Z} \mbox{$\frac i2$}\mathrm{sgn}\, k\,
\widehat \nu(k) e^{it k}.
\end{split}
\]
Since $\nu \in L^1(\torus)$ the coefficients are bounded and go to zero
and thus the analytic function $F$ has a representation
\[
F(z) = C + \puo \widehat\nu(0) + \sum_{k=1}^\infty \widehat\nu(k) z^k
\]
For Grenander--Szeg\H{o} result we need $F(0) = u(0)$ and since $u$ has
the sphere averaging property, we know that
\[
u(0) = \torusInt \puo \nu(e^{it}) \di t = \puo\widehat \nu(0)
\]
which means that $C = 0$. Therefore, we can define $q$ as the radial
limit of $z \mapsto \exp(F(z))$ which coincides with
\[
q(t) = \exp\big(\puo \nu(e^{it}) + i/2 \hilbert
\nu(e^{it})\big)
= \sqrt{f(t)} \exp\big(i/2 \hilbert(\log f(t))\big)
\]
\end{proof}
The Lemma~\ref{lemma:C:L2} provides the analytic square root for the
reciprocal of the symbol $g_\alpha$ and it follows from
Lemma~\ref{lemma:C:L1}.
\begin{proof}[Proof of Lemma~\ref{lemma:C:L2}]
We know by Lemma~\ref{lemma:C:L1} that $q_\alpha$ is of form
\[
t \mapsto \rho_1(e^{it})\exp\big(-\puo(I + i\hilbert)\log g_\alpha(t)\big)
\]
for some inner function $\rho_1$. Since $w_\alpha \overline{w_\alpha} =
  \theta_{2\alpha}$, we know by Lemma~\ref{lemma:C:L1} that
\[
w_\alpha(t) = \rho(e^{it}) \exp\big(\puo(I + i\hilbert)\log \theta_{2\alpha}\big)
\]
where $\rho$ is an inner function. We define $q_\alpha$ by choosing
$\rho_1 = \rho$. This means that
\[
q_\alpha(t)/w_\alpha(t) = \exp\big(\puo(I + i\hilbert)\log
(\theta_{2\alpha}(t)g_\alpha^{-1}(t)\big) = r_\alpha.
\]
\end{proof}
The Lemma~\ref{lemma:C:L4} gives the asymptotics of the Fourier
coefficients of $\psi_\alpha$ and it also gives the asymptotics of the
related functions via the mapping properties of Hilbert transform.
\begin{proof}[Proof of Lemma~\ref{lemma:C:L4}]
We notice that 
\[
\theta_{2\alpha}(t) = \abs{t}^{2\alpha} (1 + c_1 t^2 + t^4\varphi_1(t))
\]
for a certain $\varphi_1 \in C^\infty$ and
\[
g_{\alpha}(t) = \abs{t}^{-2\alpha} \varphi_{2,\alpha}(t) +
t^2\varphi_{3,\alpha}(t) 
\]
for a certain $C^\infty(\torus)$ function $\varphi_{3,\alpha}$ such that
$\varphi_{3,\alpha}(0) > 0$ and where on the cut-off function
$\varphi_{2,\alpha} \in C^\infty(\torus)$ with support in $(-1,1)$ and
$\varphi_2(t) = 1$ for $t \in (-\puo,\puo)$.
This implies that $u = \log
(\theta_{2\alpha} g_\alpha)$ is a $C^1(\torus)$-function and the second weak
derivative is in $L^1(\torus)$ and moreover,
\[
u''(t) = c_\alpha \abs{t}^{2\alpha}\varphi_{3,\alpha}(t) + u_2(t)
= c_\alpha \varphi_{3,\alpha}(0) g_{-\alpha}(t)  + u_3(t)
\]
where $c_\alpha = (2\alpha+2)(2\alpha+1)$ and $u_2$ and $u_3$ are
certain $C(\torus)$-functions with integrable derivative. This implies
that $\abs{\widehat u_3(k)} \lesssim k^{-1}$ for every $\alpha \in
(-\puo, \puo)$.

Moreover, since
$\widehat g_{\alpha}(k) \asymp c_{2,\alpha} k^{2\alpha-1}$, we have already
shown that $\widehat u = c_{3,\alpha} k^{-3-2\alpha} +
\littleOh{k^{-3-2\alpha}}$ in the case when $\alpha < 0$. If $\alpha >
0$ we need to first differentiate $u_3$ and since
\[
u_3'(t) = c_{4,\alpha} \abs t^{2\alpha}\varphi'_{3,\alpha}(t) + u_4(t)
\]
we deduce that $\widehat u_3 \in \bigOh{k^{-2}}$ and the 
$\widehat u = c_{3,\alpha} k^{-3-2\alpha} + \littleOh{k^{-3-2\alpha}}$
holds for $\alpha \in (0,\puo)$ as well. If we continue differentiation
and removal of the leading singularity, we obtain a full asymptotic
expansion of $\widehat u$.

From the existence of the asymptotic expansion of $\widehat u$, we see
that $\widehat w(k) := \widehat u(k-1) - 2\widehat u(k) + \widehat u(k+1) \asymp
c_{5,\alpha} k^{-5-2\alpha}$ where $w(t) = (1-\cos t) u(t)$. This
implies that $(1-\cos t) \hilbert u$ is smoother than $\hilbert u$ and
therefore, the zero is the only point that gives a contribution to the
Fourier series of $v := \puo(I+\hilbert)u$ and thus, 
\[
\widehat r(k) = e^v(0) v(0)^{-1} k^{-2} \widehat v(k) \asymp
k^{-3-2\alpha}.
\]
Another but more complicated way is to apply
Bojanic--Karamata Tauberian Theorem (\cite[Theorem 4.3.2]{MR898871}) to
deduce this fact.
\end{proof}
In Lemma~\ref{lemma:C:L5} we compute the asymptotics of the analytic
square root of the pure Fisher--Hartwig symbol using the explicit representation of
$w_\alpha$.
\begin{proof}[Proof of Lemma~\ref{lemma:C:L5}]
Since $w_\alpha(t) = (1-e^{it})^{\alpha}$ and since
\[
(1-z)^{\alpha} = \sum_{k = 0}^\infty (-1)^k\binom \alpha k z^k
\]
we can deduce that
\[
\widehat w_\alpha(k) = (-1)^k\binom \alpha k = (-1)^k\frac{\alpha^{\underline
k}}{k!} 
= \frac{(k-1-\alpha)^{\underline k}}{k!}
= -\alpha \prod_{j=2}^{k} \big(1 - \frac{\alpha+1}{j}\big)
\]
Since $\alpha + 1 \in (\puo, \tpuo)$, it holds that $\abs{j^{-1}(\alpha
+ 1)} < 1$ for every $j \ge 2$ and so
\[
\log(-\alpha^{-1} \widehat w_\alpha(k)) = \sum_{j=2}^k \log (1 -
\mbox{$\frac\beta j$}) = -\beta\sum_{j=2}^k j^{-1} - \sum_{j=2}^k \sum_{n=2}^\infty
\frac{\beta^n}{nj^n}
\]
Therefore,
\[
\log(-\alpha^{-1} \widehat w_\alpha(k)) = -\beta\int_1^k t^{-1} \di t +
C_\beta + \bigOh {k^{-1}} = \log k^{-\beta} + C_\beta + \bigOh{k^{-1}}
\]
for some constant $C_\beta$.
However, since 
\[
\widehat w_\alpha(k) = (-1)^k\binom \alpha k = (-1)^k\frac{\alpha^{\underline
k}}{k!} 
= \frac{(k-\beta)^{\underline
k}}{k!} = k^{-\beta}\Big(\frac{k! k^{-\beta}}{(k-\beta)^{\underline
k}}\Big)^{-1}
\]
and by Gauss limit formula for Gamma function, we have
\[
\lim_{k\to \infty}\widehat w_\alpha(k)k^{\beta}
= \lim_{k \to \infty}\Big(\frac{k! k^{-\beta}}{(k-\beta)^{\underline
k}}\Big)^{-1} = \Gamma(-\beta)^{-1}
\]
which is valid every $-\beta \notin \Z$. Therefore,
\[
\widehat w_\alpha(k) = \Gamma(-\beta)^{-1} k^{-\beta}(1+\bigOh{k^{-1}})
\]
and the claim follows.
\end{proof}
The last piece (Lemma~\ref{lemma:C:L6}) combines the previous lemmata
with a straightforward convolution argument.
\begin{proof}[Proof of Lemma~\ref{lemma:C:L6}]
Since the convolution of Fourier transforms is the Fourier transform of
the product, we notice that claim is equivalent with
\[
\widehat {w_\alpha \nu} (k) = \bigOh k^{-2-\alpha}
\]
where $\nu (t) = r_\alpha(t) - r_\alpha(0)$. Furthermore,
\[
\widehat\nu(k) = \Iverson{k \ne 0} \widehat r_\alpha(k)
\]
and so
\[
\abs{\widehat w_\alpha * \widehat \nu (k)} = \Big\lvert\sum_{j = 1}^k \widehat
w_\alpha(k-j) \widehat r_\alpha(k)\Big\rvert \lesssim k^{-3-2\alpha} + \sum_{j=1}^{k-1}
(k-j)^{-1-\alpha} j^{-3-2\alpha}.
\]
The last sum can be estimated with an integral
\[
\int_1^{k-1} (k-t)^{-\beta} t^{-1-2\beta} \di t =
k^{-3\beta}\int_{h}^{1-h} (1-s)^{-\beta} s^{-1-2\beta} \di s
\]
where $\beta = \alpha + 1$.
The integral on the right is integrable at zero for every $\beta \in
(\puo, \tpuo)$ but it is integrable at one only when $\beta \in
(\puo,1)$. Therefore, when $\beta \in (\puo,1)$ we have an estimate
\[
\abs{\widehat w_\alpha * \widehat \nu (k)} \lesssim k^{-3-2\alpha} +
k^{-3\beta} \asymp k^{-3- 3\alpha} \le k^{-2-\alpha}
\]
for $k > 0$, since $-1-2\alpha \le 0$. When $\beta \in (1,\tpuo)$ we estimate
\[
\int_{h}^{1-h} (1-s)^{-\beta} s^{-1-2\beta} \di s
\asymp 
\int_{\puo}^{1-h} (1-s)^{-\beta} \di s \asymp h^{1-\beta} = k^{\beta-1}
\]
and therefore,
\[
\abs{\widehat w_\alpha * \widehat \nu (k)} \lesssim k^{-3-2\alpha} +
k^{-1-2\beta} \asymp k^{-3- 2\alpha} \le k^{-3} \le k^{-2-\alpha}
\]
since now $\alpha > 0$.

\end{proof}
\section{Proofs of auxiliary results in Section~\ref{USLLN}}
\label{PFD}

This section is dedicated to the technical results that were postponed
in Section~\ref{SLLNGQF}. We begin with the almost trivial proof of
Lemma~\ref{lemma:ext:l1}.
\begin{proof}[Proof of Lemma~\ref{lemma:ext:l1}]
  This follows from the two observations.
  \begin{itemize}
    \item[$i)$] $\partial_\alpha (A_\alpha B_\alpha) = (\partial A_\alpha)
    B_\alpha + A_\alpha \partial_\alpha B_\alpha$ for every
    differentiable $A_\alpha$ and $B_\alpha$
    \item[$ii)$] $\partial_\alpha (A_\alpha A_\alpha^{-1}) = 0$ for
    every invertible and differentiable $A_\alpha$
  \end{itemize}
  Using these and simple algebra the claim follows.
\end{proof}
Next we show the Lemma~\ref{lemma:ext:lu1} which
combines~\ref{lemma:ext:l1} with analysis of the symbol $g_\alpha$.
\begin{proof}[Proof of Lemma~\ref{lemma:ext:lu1}]
  The part $(1)$ follows by the fact that the symbols $g_\alpha$ are
  uniformly bounded from below by $1/\lambda_1 > 0$ when $\alpha > 0$. 
  
  The part $(2)$ follows by differentiating the symbol $g_\alpha$ and
  noticing that $\partial_\alpha g_\alpha$ are outside a neighbourhood of
  zero uniformly bounded from above by $\lambda_2 > 0$. In the
  neighbourhood of zero $\partial_\alpha g_\alpha(t) \asymp |t|^{-2\alpha}
  \log |t|^{-1}$ and we can choose a uniform neighbourhood where this
  holds.

  The part $(3)$ follows from $(1)$, $(2)$ and Lemma~\ref{lemma:ext:l1}
  since
  \[
  \begin{split}
  \langle \partial_\alpha T_n(g_\alpha)^{-1} z , z \rangle
  & = \langle -\partial_\alpha T_n(g_\alpha) w_\alpha , w_\alpha \rangle
  \le 0 
  \end{split}
  \]
  where $w_\alpha = T_n(g_\alpha)^{-1} z$.
\end{proof}
The proof of Lemma~\ref{lemma:ext:lu2} is very similar to the proof of 
Lemma~\ref{lemma:ext:l1}.
\begin{proof}[Proof of Lemma~\ref{lemma:ext:lu2}]
  The part $(1)$ follows by comparing the symbols $g_\alpha$. At the
  neighbourhood of the zero the $g_\alpha(t) \asymp t^{2|\alpha|} \ge
  t^{2|\gamma|} \asymp g_\gamma(t)$. Since outside the origin the symbols
  are uniformly bounded from above and below, we can choose $\lambda_3 > 0$ so
  that $g_\alpha \ge 1/\lambda_3 g_\gamma$.

  The part $(2)$ follows by analysing the derivative $\partial_\alpha g$
  of the symbol. In the neighbourhood of zero the $\partial_\alpha
  g_\alpha(t) \asymp t^{2|\alpha|} \log |t|^{-1}$ is strictly positive
  when $t \ne 0$ and zero when $t=0$. Outside the neighbourhood of zero
  the functions $\partial_\alpha g_\alpha$ change the sign but are
  uniformly bounded from above. Therefore, we can choose $\lambda_4 > 0$ such
  that $\partial_\alpha g_\alpha \ge -\lambda_4 g_\alpha$.

  Part $(2)$ uses again Lemma~\ref{lemma:ext:l1} and parts $(1)$ and
  $(2)$.
  Denoting $w_\alpha = T_n(g_\alpha)^{-1}$ we get
  \[
  \begin{split}
  \langle \partial_\alpha T_n(g_\alpha)^{-1} z , z \rangle
  & = \langle -\partial_\alpha T_n(g_\alpha) w_\alpha , w_\alpha \rangle
   \le \lambda_4 \langle {z} , {w_\alpha} \rangle \\
   &\le \lambda_3 \lambda_4 \langle z , T_n(g_\gamma)^{-1}z \rangle
  \end{split}
  \]
  as claimed.
\end{proof}
\begin{proof}[Proof of Lemma~\ref{lemma:est:lu1}]
Let us first choose a $\lambda'_0 > 0$ such that
$\partial_\alpha g_\alpha(\lambda) \ge 0$ for every $\abs{\lambda} \le
\lambda'_0$ and for every $\alpha < 0$ which is possible since in the
neighbourhood of zero the derivative $\partial_\alpha g$ behaves like
$t^{-2\abs{\alpha}} \log \abs{t}^{-1}$. After choosing such a
$\lambda'_0 > 0$ we can compute $\lambda''_0 > 0$ which is the
$\mu'_0 := \inf\set{g_\alpha(\lambda)}{\abs{\lambda} >
\lambda'_0,\alpha < 0}$. If $\sup_{\alpha} g_\alpha(\lambda'_0) \le
\mu'_0$ we define $\lambda_0 = \lambda_0'$ otherwise we just take
$\lambda_0 < \lambda_0'$ so small that that 
$\sup_{\alpha} g_\alpha(\lambda_0) \le \mu'_0$ which is possible since
the $g_{\alpha}$ are equicontinuous in the neighbourhood of origin for
$\alpha \in [-\puo+\e, -\e]$.

The auxiliary symbol $\widetilde g_\alpha(\lambda) :=
g_\alpha(\lambda \wedge \lambda_0)$ is now seen to satisfy the
conditions for large enough $\lambda_5 > 0$.
\end{proof}
\section{Proofs of auxiliary results in Section~\ref{param-estimation}}
\label{PFE}
\begin{proof}[Proof of Lemma~\ref{mean}]
  We already know that $|\alpha_n - \wh H| \le M/\log n$ and therefore, we can
  use an ansatz $\alpha_n = \wh H + \theta$. Since $n^{-2\theta} = \fii_n(\alpha_n)$ we can use
  Taylor expansions for $\wh F$ and $\wh F'$ around $\wh H$ and we get
  an equation
  \[
  n^{-2\theta} = 1 + \wh F'(\wh H)\theta + \frac{\wh F'(\wh H)}{2\log n} + \bigOh
  \theta^2 + \bigOh \theta (\log n)^{-1}.
  \]
  The apriori estimate $\theta = \bigOh (\log n)^{-1}$ reduces this to 
  \begin{equation}
    \label{eq22}
    n^{-2\theta} = 1 + \wh F'(\wh H)\theta + \frac{\wh F'(\wh H)}{2\log n} + \bigOh
    (\log n)^{-2}.
  \end{equation}
  We can now take logarithms, and use the apriori estimate for terms
  $\bigOh {\theta^2}$ and $\bigOh {\theta (\log n)^{-1}}$. Therefore, the
  asymptotic representation~\eqref{eq22} reduces to an asymptotic linear
  equation for $\theta$ namely,
  \[
  \theta(-2 \log n - \wh F'(\wh H)) = \frac{\wh F'(\wh H)}{2\log n} + \bigOh
  (\log n)^{-2}
  \]
  which proves the claim.
\end{proof}
\begin{proof}[Proof of Lemma~\ref{lemma77}]
  We already have that
  \[
  \kappa_n''(\alpha) = -n^{2(\alpha - \wh H)+1}(\log n)^2 \psi_n(\alpha).
  \]
  We also know that 
  $n^{2(\alpha_n - \wh H)} = 1 / \fii_n(\alpha_n)$.
  Therefore,
  \[
  \kappa_n''(\alpha) = -n {(\log n)^2}n^{2(\alpha - \alpha_n)}
  \frac{\psi_n(\alpha)}{\fii_n(\alpha_n)}.
  \]
  From this we immediately compute that
  \[
  \kappa_n^{(3)}(\alpha) = -n{(\log n)^2}n^{2(\alpha - \alpha_n)}
  \frac{2 \log n \psi_n(\alpha) + \psi_n'(\alpha)}{\fii_n(\alpha_n)}
  \]
  and
  \[
  \kappa_n^{(4)}(\alpha) = -n{(\log n)^2}n^{2(\alpha - \alpha_n)}
  \frac{{4 (\log n)^2  \psi_n(\alpha)}+ 4\log n\psi_n'(\alpha)
  + \psi_n''(\alpha)}{\fii_n(\alpha_n)}.
  \]
  Since
  \[
  \psi_n(\alpha) = 2\fii_n(\alpha) + \fii_n'(\alpha)(\log n)^{-1}
  \]
  and for $j \le 3$ we have
  \[
  \| \fii_n^{(j)} \|_\infty = \bigOh 1
  \]
  it follows that
  \[
  | \kappa_n^{(4)} (\alpha) | \le 8 n (\log n)^4
  \]
  for $|\alpha - \alpha_n| \ll (\log n)^{-1}$. For second and third derivatives we
  get
  \[
  \kappa_n^{(2)} (\alpha_n) = -2 n(\log n)^2(1+ \bigOh {(\log n)^{-1}})
  \]
  and 
  \[
  \kappa_n^{(3)} (\alpha_n) = -4 n(\log n)^3(1+ \bigOh {(\log n)^{-1}})
  \]
\end{proof}
\begin{proof}[Proof of Lemma~\ref{tail_estimates}]
The remainder part is immediately estimated by 
\[
J(n) = \Iverson{\gamma > 0} \frac{e^{-K_n(\wh F)(\alpha(n))}}
{n \log n}{\bigOh{e^{\gamma n \log n}}}.
\]
where $\gamma := \wh H - \puokki$.

The lower tail and the upper tail calculations are essentially the same
so we only do the lower tail, so we assume that $\wh V = (\gamma, \beta_-(n)]$
for some $\beta_-(n) < \alpha(n)$.  We re-express the integral
\[
\int_\gamma^{\beta_-(n)} e^{\kappa_n(\alpha)} \di \alpha = 
\int_\gamma^{\beta_-(n)} \frac{\di{e^{\kappa_n(\alpha)}}}{\kappa_n'(\alpha)} 
\]
Lemma~\ref{concave} implies that $\kappa_n'(\alpha)$ is monotonically decreasing and thus
$1/\kappa_n'$ is a monotonically increasing function and since $\kappa_n'$ has a
unique zero point at $\alpha_n$, the division is well defined. Therefore,
\[
\int_\gamma^{\beta_-(n)} e^{\kappa_n(\alpha)} \di \alpha \le \frac 
{{e^{\kappa_n(\beta_-(n))} - e^{\kappa_n(\gamma)}}}
{\kappa_n'(\beta_-(n))}
\]
This implies the claim.
\end{proof}
\begin{lem}
\label{derivative}
  We have that
  \[
  \kappa_n'(\alpha_n + \theta) = -2 \theta n(\log n)^2 (1 + \littleOh 1)
  \]
  for large enough $n$ and $\theta \ll (\log n)^{-1}$.
\end{lem}
\begin{proof}[Proof of Lemma~\ref{derivative}]
  By substitution we have
  \[
  \kappa_n'(\alpha_n+ \theta) = n\log n\bpa{1- n^{2(\alpha_n - \wh H)} n^{2\theta} \psi_n(\theta)}
  \]
  where $\psi_n(\theta) = \fii_n(\alpha_n + \theta)$. Since $\kappa_n'(\alpha_n) = 0$ we have
  by substitution that 
  \[
  1 = n^{2(\alpha_n - \wh H)}\psi_n(0).
  \]
  Therefore,
  \[
  \kappa_n'(\alpha_n+ \theta) = n\log n\bpa{1- n^{2\theta} \mu_n(\theta)}
  \]
  where $\mu_n(\theta) := \psi_n(\theta)/\psi_n(0)$. Since $\mu_n(0) = 1$ and $\mu_n'$ is
  \[
  \mu_n'(\theta) = \frac{\fii_n'(\alpha_n + \theta)}{\psi_n(0)} = \frac{\wh F'(\alpha_n +
  \theta)}{\psi_n(0)} + \bigOh (\log n)^{-1}
  \]
  we have by the Taylor expansion that
  \[
  \mu_n(\theta) = 1 + \theta (\eta + \bigOh (\log n)^{-1}) + \bigOh \theta^2
  \]
  where $\eta = (\wh F'/\wh F) (\wh H)$. Hence, we obtain an representation
  \[
  \kappa_n'(\alpha_n+ \theta) = n\log n\bpa{1- (1 + {2\theta} \log n + \bigOh \theta^2 (\log
  n)^{2}) \mu_n(\theta)}
  \]
  which simplifies to
  \[
  \kappa_n'(\alpha_n+ \theta) = n\log n\bpa{- {2\theta} \log n + \bigOh \theta(\log n)^{-1} + \bigOh \theta^2 (\log
  n)^{2}}
  \]
  which is equivalent with
  \[
  \kappa_n'(\alpha_n+ \theta) = -2\theta n(\log n)^2 \bpa{1 + \bigOh {\big( (\log n)^{-2} +
  |\theta| (\log
  n)\big)}}.
  \]
\end{proof}

\begin{proof}[Proof of Lemma~\ref{tail_error}]
  Since $\theta = \e_n n^{-\nicefrac1{2}}(\log n)^{-1} \ll (\log n)^{-1}$ we can use 
  Lemma~\ref{derivative} and we obtain
  \[
  \kappa_n'(\alpha_n \pm \theta) \asymp \mp \theta n(\log n)^2 = \mp \e_n n^{\nicefrac1{2}} \log n.
  \]
  This together with Lemma~\ref{tail_estimates} gives the claim.
\end{proof}
\begin{proof}[Proof of Lemma~\ref{main}]
We use the Taylor expansion on $k_n$ around $0$. This gives
\[
k_n(\alpha) = \frac 12 k_n''(0) \alpha^2 + \frac 16 k_n^{(3)}(0) \alpha^3 +
\bigOh(k_n^{(4)}(0) \alpha^4).
\]
Using the change of variable $\alpha' = \tau_n(\alpha)$ we get that
\[
\int_{\beta_-(n)}^\beta e^{k_n(\alpha-\alpha(n))} \di \alpha =
\frac 1 {\log n \sqrt {n c_n}} \int_{\tau_-(n)}^{\tau(\beta)} e^{\psi_n(\alpha )} \di
\alpha
\]
where $\psi_n(\alpha) = k_n(\alpha (\log n)^{-1} n^{-1/2} c_n^{-1/2})$. Combining
this with the Taylor expansion we see that
\[
\psi_n(\alpha) = -\frac 12 \alpha^2 + \bigOh (n^{-1/2} |\alpha|^3) = -\frac
12 \alpha^2 (1 +
\bigOh n^{-\nicefrac1{2}} \e_n)
\]
where we used the facts that $k_n''(0) = -c_n n(\log n)^2$ and
the estimate that $k_n'''(0) \lesssim n(\log n)^3$ and the estimate
$|\alpha|^3 \le c_n |\alpha|^2 \e_n$.

From this we obtain both the estimate from the above and from the below
for the integral of $e \circ \psi_n$ and therefore,
\[
\frac 1 {\sqrt {2\pi}}\int_{\beta_-(n)}^\beta e^{k_n(\alpha-\alpha(n))} \di \alpha =
\frac 1 {\log n \sqrt {n c_n}} (\varPhi \circ \tau(\beta) - \lambda_-(n)) (1 + \bigOh {n^{-\nicefrac1{2}}\e_n}).
\]
\end{proof}

\end{document}